\RequirePackage{fix-cm}
\documentclass[smallextended,oneside]{svjour3}

\smartqed

\usepackage{graphicx}
\usepackage{amsmath,amssymb}
\usepackage{subcaption}
\usepackage[ruled,vlined]{algorithm2e}
\usepackage{booktabs}
\usepackage{tikz}
\usetikzlibrary{quotes,arrows.meta,positioning,3d}
\usepackage{array}

\newcommand{\subheading}[1]{\noindent{{\bfseries\sffamily #1}}}
\journalname{Journal of Scientific Computing}
\begin{document}

\title{Enhancing Future Prediction of Linear and Nonlinear Reduced-Order Models for Transport-Dominated Problems Using Lagrangian Data}

\author{Meng Li \and Yang Xiang\and Zhichao Peng}


\institute{Meng Li \at
Department of Mathematics, The Hong Kong University of Science and Technology, Clear Water Bay, Hong Kong Special Administrative Region of China\\
\and
Zhichao Peng (corresponding author)\at
Department of Mathematics, The Hong Kong University of Science and Technology, Clear Water Bay, Hong Kong Special Administrative Region of China \\
\email{pengzhic@ust.hk}
\and
Yang Xiang (corresponding author) \at
Department of Mathematics, The Hong Kong University of Science and Technology, Clear Water Bay, Hong Kong Special Administrative Region of China, and HKUST Shenzhen-Hong Kong Collaborative Innovation Research Institute, Futian, Shenzhen, China \\
\email{maxiang@ust.hk}
}

\date{Received: date / Accepted: date}

\maketitle

\begin{abstract}
Designing effective reduced-order models (ROMs) for parametrized transport-dominated problems remains challenging because of the well-known Kolmogorov barrier. Autoencoder-based nonlinear ROMs have been developed to improve the compression ability for such systems. However, despite their stronger compression ability, autoencoder-based ROMs constructed in the Eulerian frame may fail to accurately predict future solutions, due to the poor coherence between historical and future solutions in the Eulerian frame.

In contrast, we show that representing transport-dominated dynamics in the Lagrangian frame can lead to a significantly faster decay of the Kolmogorov $n$-width and improve coherence between historical and future solutions. Building on these insights, we develop two non-intrusive ROMs leveraging Lagrangian data: a Lagrangian autoencoder-based ROM and a Lagrangian parametric dynamic mode decomposition. Numerical experiments demonstrate that these Lagrangian ROMs achieve more accurate and stable future predictions than their Eulerian counterparts.

\end{abstract}

\section{Introduction\label{sec:intro}}
Parametrized partial differential equations (PDEs) arise in a wide range of multi-query scientific and engineering applications, such as sensitivity analysis, uncertainty quantification, and optimal control. In these applications,  high-fidelity discrete systems, namely full-order models (FOMs), are solved repeatedly, demanding significant computational costs. To lower costs, reduced-order models (ROMs) are constructed to provide efficient low-dimensional surrogates that approximate high-fidelity solutions. Typically, an offline–online decomposition framework is applied in model order reduction. In the offline stage, a low-dimensional representation is constructed by extracting coherent structures from collected solution snapshots; in the online stage, fast predictions are enabled via projection or interpolation using the low-dimensional representation constructed offline.

Despite their success for elliptic and parabolic problems, linear ROMs often face a fundamental limitation in transport-dominated regimes. In such settings,  the optimal error of approximating the solution manifold $\mathcal{M}$ by an $n$-dimensional linear subspace $\mathcal{V}_n$, namely the Kolmogorov $n$-width,
\begin{equation}\label{eq:n-width}
d_n(\mathcal{M}):= \inf_{\substack{\dim(\mathcal{V}_n) = n}} \sup_{u\in\mathcal{M}} \inf_{v \in \mathcal{V}_n} \|u - v\|,
\end{equation}
may decay slowly, leading to the well-known Kolmogorov barrier \cite{Ohlberger2016-limitations,Peherstorfer2022-kolmogorov-barrier}. As a consequence, relatively high-dimensional ROMs may be required to achieve the desired approximation accuracy.

To mitigate this limitation, various nonlinear ROMs have been developed. One strategy is to construct piecewise-linear ROMs by introducing offline and/or online adaptivity in space and/or time \cite{Carlberg2015a-h-adaptive,peherstorfer2020model,huang2023predictive,dihlmann2011model,san2015principal,ji2025aaroc,jin2025adaptive,li2025localized}. Another strategy is to construct nonlinear ROMs through nonlinear transformations that capture intrinsic low-dimensional structures in transport-dominated problems. Such transformations can be constructed by using characteristic lines \cite{rowley2003reduction,Mojgani2017-lagpod,reiss2018shifted,Lu2020-lagdmd,papapicco2022neural}, registration \cite{taddei2015reduced,Mojgani2017-lagpod,taddei2020registration}, problem-specific structures \cite{rim2018transport,cagniart2018model}, optimal transport \cite{ehrlacher2019nonlinear,rim2023manifold}, and deep learning \cite{Lee2020-cae-rom,kim2022fast,fresca2021comprehensive,papapicco2022neural,peng2023learning}. One representative and widely used method among these transformation-based approaches is the autoencoder-based ROM \cite{Lee2020-cae-rom,fresca2021comprehensive,kim2022fast}. For a comprehensive overview, we refer readers to \cite{Peherstorfer2022-kolmogorov-barrier}.

Despite their effectiveness in compressing transport-dominated solutions, current autoencoder-based nonlinear ROMs trained in a fixed Eulerian frame may struggle to accurately extrapolate in future prediction tasks for transport-dominated systems. Due to the moving nature of solutions to transport-dominated problems, desired future solutions in prediction tasks may become weakly correlated or even uncorrelated with the historical training data. Hence, future prediction is a challenging out-of-distribution extrapolation problem.

In this paper, we argue that this limitation can be addressed by exploiting the characteristic-line information encoded in the Lagrangian frame. We show that the Kolmogorov $n$-width of some transport problems decays faster in the Lagrangian frame. Moreover, unlike the Eulerian frame, future solution data in the Lagrangian frame remain  strongly correlated with the historical training data. In this sense, using Lagrangian data  converts future prediction from an out-of-distribution extrapolation into an interpolation over correlated data. These insights explain the success of Lagrangian POD \cite{Mojgani2017-lagpod} and Lagrangian DMD \cite{Lu2020-lagdmd,lu2021dynamic}.  

Building on this observation, we further develop Lagrangian ROMs that can more effectively predict future solutions of transport-dominated problems compared to their Eulerian counterparts. First, we design a Lagrangian autoencoder. In the offline stage, Lagrangian data are used to train an autoencoder that learns a Lagrangian low-dimensional nonlinear manifold approximation. In the online stage, future predictions are performed in a non-intrusive manner using  parametric DMD framework proposed in \cite{Andreuzzi2023-pdmd-partition-stack,Duan2024-cae-hodmd}. In addition, as a byproduct of this non-intrusive online stage, we propose a Lagrangian parametric DMD, which can be viewed as the Lagrangian counterpart of the parametric DMD \cite{Andreuzzi2023-pdmd-partition-stack} and as a parametric extension of the Lagrangian DMD in \cite{Lu2020-lagdmd}.

The main contributions of this work are twofold. First,  we theoretically prove a faster decay of Kolmogorov $n$-widths for transport-dominated problems in the Lagrangian frame and empirically demonstrate the stronger temporal coherence between solutions. These findings imply stronger compression and prediction ability of Lagrangian ROMs. Second, we develop a Lagrangian autoencoder-based ROM and a Lagrangian parametric DMD for parametrized transport-dominated PDEs, and demonstrate their stronger prediction ability compared to their Eulerian counterparts.

The remainder of this paper is organized as follows. Section~\ref{sec:rom} reviews ROM techniques and discusses the challenges faced by linear and nonlinear Eulerian ROMs in future prediction tasks. Section~\ref{sec:motivation} motivates the use of Lagrangian data through theoretical analysis and empirical evidence. Section~\ref{sec:lagrom} presents the proposed Lagrangian autoencoder and Lagrangian parametric DMD. Section~\ref{sec:experiment} demonstrates the performance of the proposed methods. Finally, Section~\ref{sec:conclusion} concludes the paper and outlines potential future directions.


\section{Background and motivation}\label{sec:rom}

Constructing efficient reduced-order models (ROMs) for transport-dominated systems is challenging, especially for predicting future states beyond the training time window. Classical linear ROMs approximate the solution manifold by a low-dimensional linear subspace. The optimal approximation error over all $n$-dimensional linear subspaces is measured by the Kolmogorov $n$-width defined in~\eqref{eq:n-width}. Unfortunately, this optimal error often decays slowly in transport-dominated problems, implying that a relatively high-dimensional ROM is required to achieve accurate approximation \cite{Ohlberger2016-limitations,ehrlacher2019nonlinear,Peherstorfer2022-kolmogorov-barrier}. This phenomenon, known as the Kolmogorov barrier, limits the efficiency of linear ROMs.

To alleviate this limitation, various nonlinear ROMs have been developed to approximate the solution manifold using nonlinear representations. Among them, autoencoder-based ROMs, which learn low-dimensional nonlinear manifold approximations from data, have shown strong compression performance for transport-dominated systems \cite{Lee2020-cae-rom,fresca2021comprehensive,kim2022fast}. 

However, in this section, we demonstrate that strong compression in the Eulerian frame alone is insufficient for an accurate future prediction, thereby motivating the use of the Lagrangian frame. To serve this purpose, we first introduce the model problem in both the Eulerian and Lagrangian frames, then review representative Eulerian ROMs, and finally illustrate the prediction limitations of Eulerian representations.

\subsection{Full-order model in Eulerian and Lagrangian frames}\label{subsec:fom}
Although our approach extends to higher dimensions, we begin with a one-dimensional parametrized advection--diffusion equation on $(x,t) \in [a,b]\times [0,T]$ for clarity:
\begin{subequations}\label{eqn:ad}
\begin{align}
    & \frac{\partial u}{\partial t}+f(u;\mu)\frac{\partial u}{\partial x}
    =\frac{\partial }{\partial x}\!\left(D(x,t,u;\mu)\frac{\partial u}{\partial x}\right), \quad f(u)=\frac{\partial F(u;\mu)}{\partial u}\\
    & u(x,t=0;\mu)=u_0(x;\mu)
\end{align}
\end{subequations}
with appropriate boundary conditions at $x = a$ and $x = b$. Here $\mu$ parametrizes physical configurations (e.g., material properties, forcing, or boundary conditions). This equation can be solved in either a fixed Eulerian frame or a moving Lagrangian frame.

\subheading{Eulerian frame.} In the Eulerian frame, we solve \eqref{eqn:ad} on a fixed grid $\boldsymbol{x}=[x_1,\dots,x_M]^T$ with uniform spacing $\Delta x$ and time step $\Delta t=T/N$. For simplicity, we use a first-order implicit-explicit time discretization, a first-order  upwind discretization for the advection term, and a central difference discretization for the diffusion term. Here, we omit the parameter dependence for notational simplicity.
\begin{equation}\label{eqn:euler_fdm}
\begin{aligned}
u_j^{n+1}
=\, &u_j^{n}
-\frac{\Delta t}{\Delta x}
\left(F_{j+\frac12}^{n}-F_{j-\frac12}^{n}\right) \\
&+\frac{\Delta t}{(\Delta x)^2}
\Big[
D_{j+\frac12}^{n+1}\left(u_{j+1}^{n+1}-u_{j}^{n+1}\right)
-D_{j-\frac12}^{n+1}\left(u_{j}^{n+1}-u_{j-1}^{n+1}\right)
\Big],
\end{aligned}
\end{equation}
where $D_{j+1/2}^{n+1} = \frac{1}{2}(D_{j}^{n+1}+D_{j+1}^{n+1})$, and the numerical flux is given by
\[
\begin{aligned}
    & F_{j+1/2}^n = \frac{F(u_{j+1}^n)+F(u_{j}^n)}{2}-|a^n_{j+1/2}|\frac{u^n_{j+1}-u_j^n}{2} \\
    & a^n_{j+1/2} = \left\{\begin{aligned}
        & \frac{F_{j+1}^n-F_j^n}{u_{j+1}^n-u_j^n}, \text{ if } u_{j+1}^n\neq u_j^n \\
        & f(u_j), \text{ if } u_{j+1}^n = u_j^n
    \end{aligned}\right. \\
\end{aligned}
\]

\subheading{Lagrangian frame.}
Instead of using a fixed grid, Lagrangian methods evolve grid points along characteristic lines of the transport term. Let $\hat{x}\in[a,b]$ be a reference coordinate and define the characteristic line as $\chi(\hat{x},t)$ satisfying
\begin{equation}
\frac{d\chi}{dt}=f\!\left(u(\chi(\hat{x},t),t)\right),\qquad \chi(\hat{x},0)=\hat{x}.
\end{equation}
Along the characteristic line, the solution satisfies
\begin{equation}
\frac{d}{dt}u(\chi,t)
=\left.\frac{\partial}{\partial x}\left(D(x,t,u)\frac{\partial u}{\partial x}\right)\right|_{x=\chi(\hat{x},t)}.
\end{equation}
Following \cite{Lu2020-lagdmd}, a Lagrangian numerical discretization  is given by
\begin{equation}
\left\{\begin{aligned}
    & \tilde{u}_j^n = \mathcal{I}^n_{\textrm{L}\rightarrow\textrm{E}}(u_j^n), \\
    & \tilde{u}_j^{n+1}=\tilde{u}_j^{n}
    +\frac{\Delta t}{(\Delta x)^2}\Big[D_{j+\frac12}^{n+1}(\tilde{u}_{j+1}^{n+1}-\tilde{u}_{j}^{n+1})
    -D_{j-\frac12}^{n+1}(\tilde{u}_{j}^{n+1}-\tilde{u}_{j-1}^{n+1})\Big],\\
    & u_{j}^{n+1} = \mathcal{I}^n_{\textrm{E}\rightarrow\textrm{L}}(\tilde{u}_j^{n+1}), \\
    & \chi_j^{n+1} = \chi_j^n + \frac{\Delta t}{2}(f(u_j^n)+f(u_j^{n+1})),
\end{aligned}\right.
\end{equation}
where $\tilde{u}^n$ is the solution in the Eulerian frame, $\mathcal{I}^n_{\textrm{E}\rightarrow\textrm{L}}$/ $\mathcal{I}^n_{\textrm{L}\rightarrow\textrm{E}}$ is the linear interpolation from the Eulerian/Lagrangian grid to the Lagrangian/Eulerian grid,  $\boldsymbol{\chi}^n=[\chi^n_1,\dots, \chi^n_M]^T$ is the Lagrangian grid points at  time $t_n=n\Delta t$.

\subsection{Linear and nonlinear ROMs in the Eulerian frame}\label{subsec:euler_rom}
We now briefly review two representative Eulerian ROMs and highlight their limitations for future prediction: (i) linear dynamic mode decomposition (DMD) \cite{Andreuzzi2023-pdmd-partition-stack,Huhn2023-pdmd-eigen-or-koopman}, and (ii) a nonlinear autoencoder-based ROM \cite{Lee2020-cae-rom,fresca2021comprehensive}. 

\subsubsection{Dynamic mode decomposition}\label{subsubsec:dmd}
\textbf{DMD} \cite{Schmid2010-vanilla-dmd,H.Tu2014-dmd-theory} learns a low-dimensional linear approximation to the Koopman operator $A$ for a discrete dynamical system, $u^{n+1}=Au^n$. Leveraging snapshot data, DMD extracts dominant spatial modes associated with the underlying dynamical system.

Specifically, given snapshots $\{u^k\}_{k=0}^{n}$ (at a fixed parameter value $\mu$), define
\[
\boldsymbol{U}^-=[u^0,\dots,u^{n-1}],\qquad 
\boldsymbol{U}^+=[u^1,\dots,u^{n}],
\]
a reduced-order approximation to the Koopman operator $A$ can be built as:
\begin{enumerate}
        \item Compute truncated SVD: $\mathbf{U}^{-}\approx \Phi\Sigma V^T ,\Phi\in \mathbb{R}^{n_s\times r}, r\ll n_s$.
        \item Construct a reduced-order Koopman operator: $\tilde{A} = \Phi^T U^+ V\Sigma^{-1}\in \mathbb{R}^{r\times r}$.
        \item Compute eigen-decomposition: $\tilde{A}W=W\Lambda$, $\Lambda=\mathrm{diag}(\lambda_k)$ is the eigenvalue matrix.
\end{enumerate}
Then, future solutions $u^{n+k}$ can be predicted as
\begin{equation}
    u^{n+k}= \Phi W\Lambda^k W^{-1}\Phi^T u^n.
\end{equation}

\subsubsection{A non-intrusive autoencoder-based ROM}
In this paper, we consider a non-intrusive autoencoder-based nonlinear ROM. Following \cite{Lee2020-cae-rom}, instead of seeking a linear reduced subspace, a low-dimensional nonlinear manifold is constructed by representing it as a convolutional autoencoder (CAE) neural network (its architecture illustrated in Figure \ref{fig:cae_structure}). The CAE consists of an encoder $ \mathcal{G}_e(\cdot;\theta_{\mathrm{enc}}):\mathbb{R}^{n_s} \to \mathbb{R}^{r}$ and a decoder $ \mathcal{G}_d(\cdot;\theta_{\mathrm{dec}}):\mathbb{R}^{r} \to \mathbb{R}^{n_s}$, where $r\ll n_s$ is the latent dimension, $\theta_{\mathrm{enc}}$ and $\theta_{\mathrm{dec}}$ are trainable parameters defining the autoencoder. The encoder compresses a high-dimensional full-order solution $u\in\mathbb{R}^{n_s}$ to a low-dimensional latent coordinate
\[
h = \mathcal{G}_e(u;\theta_{\mathrm{enc}}) \in \mathbb{R}^r, 
\qquad r \ll n_s.
\]
The decoder lifts a reduced-order coordinate back to the full-order solution
\[
\hat{u} = \mathcal{G}_d(h;\theta_{\mathrm{dec}})\in \mathbb{R}^{n_s}.
\] 
Given training snapshots $\{u^k\}_{k=1}^{n_t}$, the autoencoder is trained by minimizing the reconstruction error
\begin{equation}\label{eq:loss_recon}
\min_{\theta_{\mathrm{enc}},\theta_{\mathrm{dec}}}\frac{1}{n_t}
\sum_{k=1}^{n_t}
\left\|
u^k -
\mathcal{G}_d\bigl(\mathcal{G}_e(u^k;\theta_{\mathrm{enc}});\theta_{\mathrm{dec}}\bigr)
\right\|_2^2
+
\mathcal{R}(\theta_{\mathrm{enc}},\theta_{\mathrm{dec}}),
\end{equation}
where $\mathcal{R}$ denotes a regularization term.

After training, each snapshot is encoded into its latent coordinate $h^k=\mathcal{G}_e(u^k)$, producing a sequence of reduced states $\{h^k\}_k$. A non-intrusive predictor is then applied in the latent space to approximate the latent coordinate for unseen parameters. Based on the predicted latent coordinate, the prediction for the full-order solution can be obtained using the decoder.

\begin{figure}[h]
\centering
\resizebox{0.6\textwidth}{!}{\input{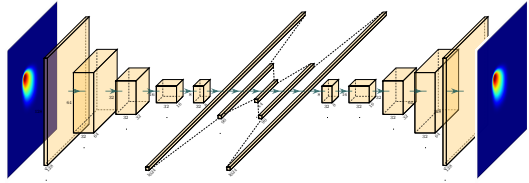}}
\caption{Conventional architecture of the convolutional autoencoder.}
\label{fig:cae_structure}
\end{figure}

\subsection{Prediction Failure in the Eulerian Frame}\label{sec:failure} 
Here, we show that both DMD and the autoencoder-based ROM may still fail in the future prediction of a transport problem.

Consider the 1D advection equation
\[
\frac{\partial u}{\partial t}+\frac{\partial u}{\partial x}=0,\quad x\in[0,2],\ t\in[0,1],
\]
with periodic boundary conditions and a localized initial pulse
\[
u_0(x)=\frac{1}{2}\exp(\;-(x-0.3)^2/0.005^2\;).
\]
Snapshots are generated from the analytic solution $u(x,t)=u_0(x-t)$ on a uniform grid with $n_s=128$ and $n_t=101$. We use $t\in[0,0.8]$ for training and predict future solutions on $t\in(0.8,1.0]$.

Figure~\ref{fig:1d_adv} compares reconstructions and forecasts from DMD and the autoencoder against the ground truth. We observe:
\begin{enumerate}
\item \textbf{Compression.} Due to the slow decay of the Kolmogorov $n$-width, DMD fails to compress the solution data effectively. With only $8$ modes, DMD fails to accurately reconstruct the training data. In contrast, using an $8$-dimensional  nonlinear approximation, the autoencoder-based ROM accurately reconstructs the training data.
\item \textbf{Prediction.} Both DMD and the autoencoder-based ROM fail to accurately predict the future solution. Future solutions become nonzero outside the support of all training snapshots. As a result, these Eulerian ROMs fail to predict translated structures outside the training support.
\end{enumerate}
In summary, due to the moving nature of transport-dominated problems, solutions at future times may exhibit structures that cannot be represented by training data in the fixed Eulerian frame. This motivates exploiting Lagrangian data to predict future solutions, where the dynamics of wave propagation is naturally encoded in the Lagrangian data representing the evolution of characteristic trajectories.
\begin{figure}[h]
    \centering
\begin{subfigure}{0.32\textwidth}
        \includegraphics[width=\textwidth]{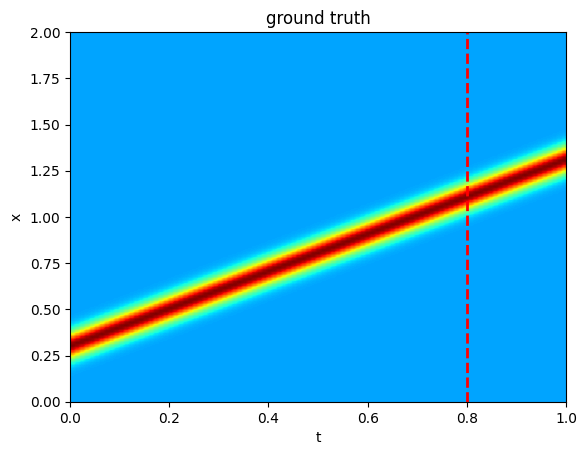}
        \caption{}\label{subfig:1d_adv_gt}
    \end{subfigure}
    \hfill
    \begin{subfigure}{0.32\textwidth}
        \includegraphics[width=\textwidth]{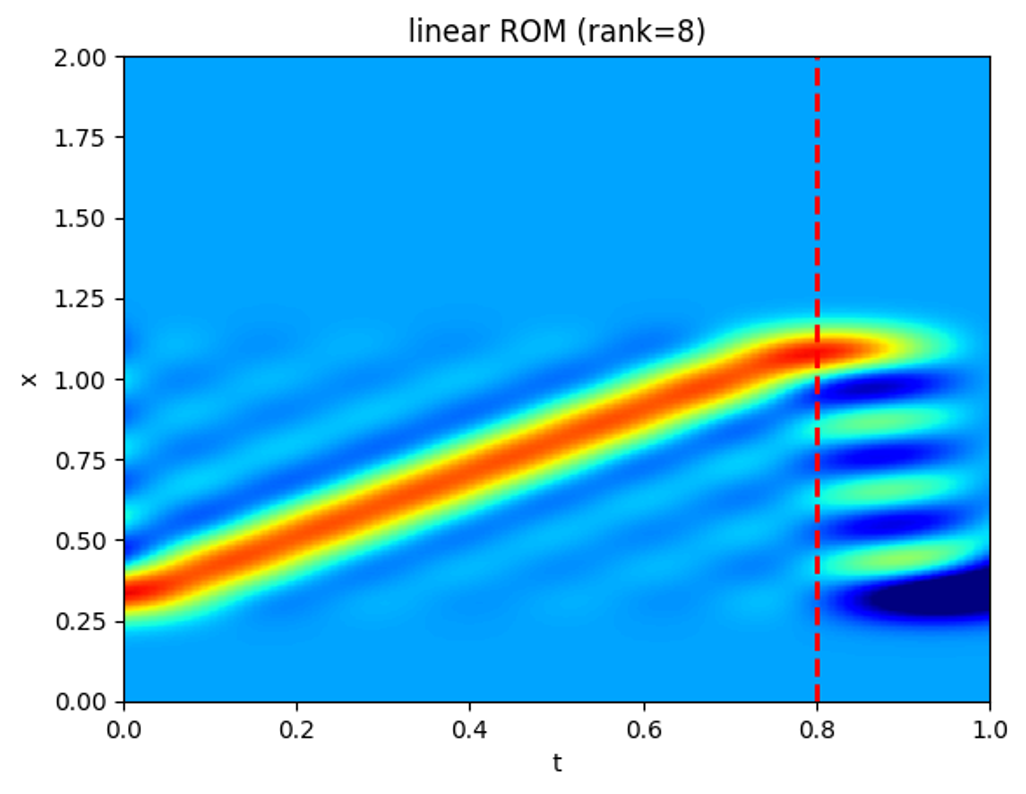}
        \caption{}\label{subfig:1d_adv_eul_dmd}
    \end{subfigure}
    \hfill
    \begin{subfigure}{0.32\textwidth}
        \includegraphics[width=\textwidth]{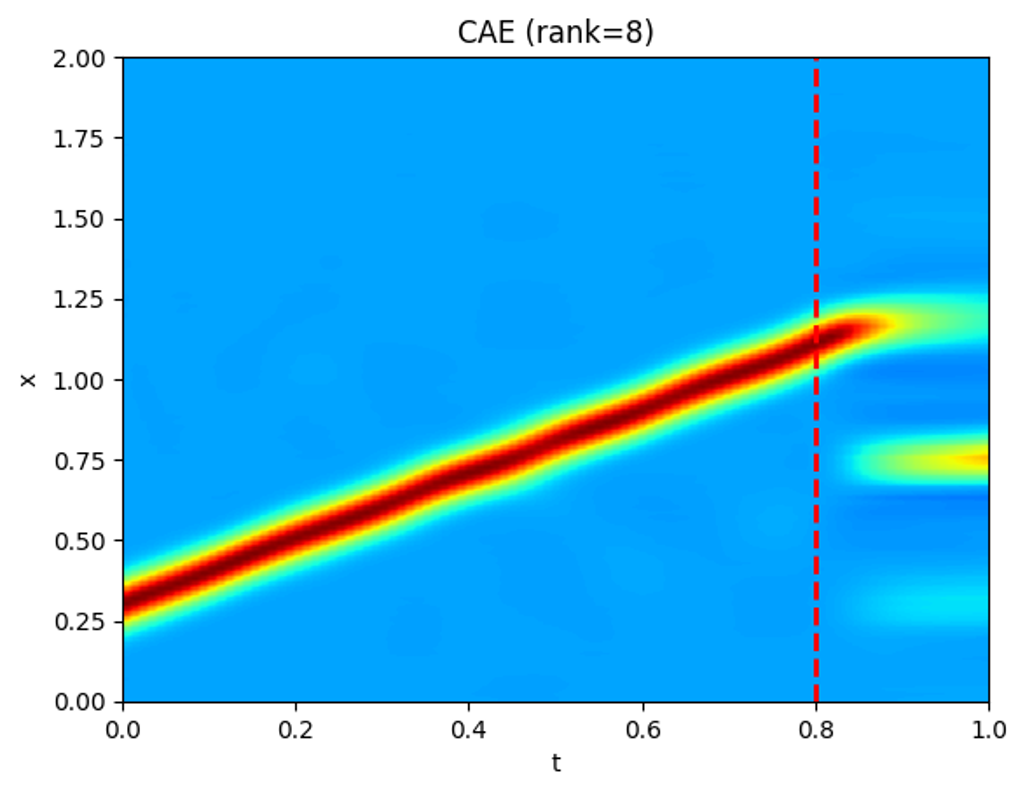}
        \caption{}\label{subfig:1d_adv_eul_cae}
    \end{subfigure}
    \caption{1D advection problem $u_t+u_x=0$. (a) Ground-truth solution. The red dashed line at $t=0.8$ separates the training window $t\in[0,0.8]$ from the test window $t\in(0.8,1]$. (b) DMD-based ROM (rank $r=8$). (c) Autoencoder-based ROM (latent dimension $r=8$). For panel (b)-(c), the region on $t\in[0,0.8]$ shows in-window reconstruction results, whereas the solution over $t\in(0.8,1]$ corresponds to out-of-window forecasts.}
    \label{fig:1d_adv}
\end{figure}


\section{Enhancing compression and prediction via Lagrangian data}\label{sec:motivation}
In this section, we motivate the use of Lagrangian data. We focus on two aspects. First, through theoretical analysis, we show that the Kolmogorov $n$-width decays significantly faster in the Lagrangian frame. Second, we demonstrate that the correlation between future solutions and historical training data is also much stronger in the Lagrangian frame. In this sense, prediction in the Lagrangian frame can be viewed as interpolation over correlated solutions, rather than out-of-distribution extrapolation.

\subsection{Kolmogorov $n$-width in the Lagrangian frame}
The Lagrangian frame, which follows characteristic trajectories, renders coherent structures approximately stationary and reduces the intrinsic dimensionality of the solution manifold. As a result, we are able to prove the faster decay of the Kolmogorov $n$-width in the Lagrangian frame for several transport-dominated problems. 

\subheading{1D transport problem.} First, we consider a linear advection problem with jump discontinuities
\begin{subequations}
\label{eq:1d-n-width-eulerian}
\begin{align}
& \frac{\partial u}{\partial t}+\frac{\partial u}{\partial x}=0,\quad  (x,t)\in[0,1]^2\\
& u(x,0)=0,\quad u(0,t)=1
\end{align}
\end{subequations}
The exact solution is $u(x,t) := u_0(x-t)=\begin{cases}
     1,x\le t, \\ 
     0,x>t. 
\end{cases} $ Let the solution manifold be the set of solutions at different times, then the corresponding Eulerian solution manifold is
$$
\mathcal{M}_E=\{u_0(x-t):t\in[0,1]\}\subset{L^2([0,1])}.
$$
As proved in \cite{Ohlberger2016-limitations}, the Kolmogorov $n$-width in the Eulerian frame with respect to $L_2$ norm $\|\cdot\|_2$, 
\begin{equation}\label{eq:1d-adv-dn}
    d_n(\mathcal{M}_E)\geq cn^{-1/2}.
\end{equation}
where $c>0$ is a constant. In contrast, we will prove that in the Lagrangian frame the $n$-width is $0$ for $n\geq 2$. In the Lagrangian frame, the governing equations become
\begin{equation}
\label{eq:1d-n-width-lagrangian}
\frac{d\chi}{dt} = 1,\quad \frac{du}{dt} = 0.
\end{equation}
where $\chi(\hat{x},t)$ is the characteristic line and $\hat{x}$ is the reference coordinate in $[0,1]$. Then the exact solution is
\begin{equation}
\begin{pmatrix}\chi(\hat{x},t)\\u(\chi,t)\end{pmatrix}= \begin{pmatrix}\hat{x}+t\\u_0(\hat{x})\end{pmatrix}\in {L^2([0,1];\mathbb{R}^2)} 
\end{equation}
endowed with the product norm $\|(\chi,u)\|_{L^2([0,1];\mathbb{R}^2)}=(\int_0^1[\chi^2(\hat{x})+u^2(\hat{x})]d\hat{x})^{1/2}$. The corresponding Lagrangian manifold is
\begin{equation}
\mathcal{M}_L=\left\{ \begin{pmatrix}\chi \\ u\end{pmatrix}: \chi=\hat{x}+t,u=u_0(\hat{x}),t\in[0,1]\right\}\subset{L^2([0,1];\mathbb{R}^2)}
\end{equation}
where time again serves as the varying parameter. In the following theorem, we prove that $\mathcal{M}_L$ is in fact $2$-dimensional. In other words, the Kolmogorov $n$-width in the Lagrangian frame satisfies $d_n(\mathcal{M}_L)=0$ for $n\ge 2$.

\begin{theorem}\label{theo:1d_tran_ko} The Kolmogorov $n$-width in the Eulerian frame satisfies 
\begin{align}\label{eqn:1d_trans_ko}
d_n(\mathcal{M}_E):= \inf_{\substack{\mathcal{V}_n}} \sup_{t\in[0,1]} \inf_{v \in \mathcal{V}_n} \|u(\cdot,t) - v(\cdot)\|\ge cn^{-1/2}.
\end{align}
where $c>0$ is a constant and 
 $\mathcal{V}_n\subset L^2[0,1]$ is an $n$-dimensional subspace.

The Kolmogorov $n$-width in the Lagrangian frame is
\begin{align}\label{eqn:1d_trans_ko_lag}
d_n(\mathcal{M}_L):=\inf_{\substack{\mathcal{V}_n}} \sup_{t\in[0,1]} \inf_{v \in \mathcal{V}_n} \left\|\begin{pmatrix}\chi(\cdot,t)\\u(\cdot,t)\end{pmatrix}-\begin{pmatrix}v_1(\cdot)\\v_2(\cdot)\end{pmatrix}\right\|=0,\quad \forall n\ge 2
\end{align}
where $\mathcal{V}_n\subset L^2([0,1];\mathbb{R}^2)$ and $t$ is the varying parameter. Here,
\[
\|v(\cdot,t)\|_{L^2([0,1];\mathbb{R}^2)}
=
\left(
\int_0^1 \bigl[v_1^2(\hat{x})+v_2^2(\hat{x})\bigr]\,d\hat{x}
\right)^{1/2}.
\]
\end{theorem}

\begin{proof}
The $n$-width in the Eulerian frame is proved in \cite{Ohlberger2016-limitations}. We focus on the Lagrangian frame here.

Consider the vector-valued function
$$
\begin{aligned}
    \begin{pmatrix}\chi \\ u \end{pmatrix}=
    \begin{pmatrix}\hat{x}+ t\\ u_0(\hat{x})\end{pmatrix}
    = \begin{pmatrix}\hat{x} \\ u_0\end{pmatrix}
    + t \begin{pmatrix} 1\\ 0\end{pmatrix}
\end{aligned}
$$
Clearly,
$$
\begin{aligned}
    \begin{pmatrix}\chi \\ u\end{pmatrix}
    \in
    \text{span}\left\{ \begin{pmatrix}\hat{x}\\ u_0\end{pmatrix},
    \begin{pmatrix}1 \\ 0\end{pmatrix} \right\}
    := \mathcal{W}
\end{aligned}
$$
Hence, $\mathcal{M}_L\subset \mathcal{W}$. Since $\dim(\mathcal{W})=2$, we have $d_n(\mathcal{M}_L)=0$ for $n\geq 2$.
\end{proof}

\subheading{2D transport problem with a parametrized transport direction.} 
Now we consider a parametrized 2D advection problem
\begin{equation}
\left\{\begin{aligned}
    & \frac{\partial u}{\partial t}+\cos\theta\frac{\partial u}{\partial x}+\sin\theta\frac{\partial u}{\partial y}=0,\quad  (x,y)\in [-1,1]^2,t\in[0,T],\theta\in[0,2\pi] \\
    & u(x,y,0) = u_0(x,y) \\
    & u(1,y,t)=u(-1,y,t)=u(x,1,t)=u(x,-1,t)=0
\end{aligned}\right.
\end{equation}
where $u_0(x,y)=\left\{\begin{aligned}1, x^2+y^2\le r^2 \\ 0,x^2+y^2> r^2 \end{aligned}\right.$ with $0<r<1/T$ and $r$ is a fixed constant.

In the Eulerian frame, the solution manifold is defined as 
\[
\mathcal{M}_E = \{u_0(x-t\cos\theta,y-t\sin\theta): \theta\in[0,2\pi],t\in[0,T]\}\subset L^2([-1,1]^2)
\]
with parameters $\theta$ and $t$. The Kolmogorov $n$-width $d_n(\mathcal{M}_E)$ depends on both the transport nature of the problem and the regularity of $u_0$ \cite{Arbes2025-kolomogorov-linear-eqn}. To illustrate this, consider the special case when $\theta = 0$. Under this assumption, the solution reduces to $u(x,y,t) = u_0(x-t,y)$, which corresponds to transport in the $x-$direction. Thus, the problem reduces to the 1D transport problem in \eqref{eq:1d-n-width-eulerian}, whose Kolmogorov $n$-width is proved to be greater than $n^{-1/2}$ in \cite{Ohlberger2016-limitations}. Consequently, by monotonicity of Kolmogorov widths with respect to set inclusion, the Kolmogorov $n$-width of $\mathcal{M}_E$ decays slower.

In the Lagrangian frame, the governing equation becomes:
\begin{equation}
\begin{cases}
&\frac{d\chi}{dt} = \cos\theta \\
&\frac{d\zeta}{dt} = \sin\theta \\
&\frac{du}{dt} = 0 
\end{cases}
\implies
\begin{cases}
&\chi(\hat{x},\hat{y},t) = \hat{x}+t\cos\theta \\
&\zeta(\hat{x},\hat{y},t) =\hat{y}+t\sin\theta \\
&u(\chi,\zeta,t) = u_0(\hat{x},\hat{y}) \\
\end{cases}
\end{equation}
where $\chi(\hat{x},\hat{y},t)$ and $\zeta(\hat{x},\hat{y},t)$ are the characteristic lines and $\hat{x},\hat{y}$ are the reference coordinates in the domain $[-1,1]^2$.

Now we define the solution manifold in the Lagrangian frame as
\[
\mathcal{M}_L =
\left\{
\begin{pmatrix}
\chi \\ \zeta \\ u
\end{pmatrix}
:
\begin{aligned}
\chi &= \hat{x}+t\cos\theta,\\
\zeta &= \hat{y}+t\sin\theta,\\
u &= u_0(\hat{x},\hat{y}), \qquad
\theta\in[0,2\pi],\ t\in[0,T]
\end{aligned}
\right\}
\subset L^2([-1,1]^2;\mathbb{R}^3)
\]
equipped with the norm $\|v(\cdot,t)\|_{L^2([-1,1]^2;\mathbb{R}^3)}=(\int_{-1}^1\int_{-1}^1[v^2_1(\hat{x},\hat{y})+v^2_2(\hat{x},\hat{y})+v^2_3(\hat{x},\hat{y})]d\hat{x}d\hat{y})^{1/2}$. $t,\theta$ are the varying parameters. Then, its Kolmogorov $n$-width becomes $0$ when $n\ge3$.

\begin{theorem}
There exists a constant $c>0$ such that for all $n\in\mathbb{N}^*$,
\begin{align}
    d_n(\mathcal{M}_E)\ge cn^{-1/2}.
\end{align}
In addition,
    \begin{align}
        d_n(\mathcal{M}_L)=0,\quad \forall n \ge 3
    \end{align}
\end{theorem}

\begin{proof}

In the Eulerian frame, we have
\[
\begin{aligned}
d_n(\mathcal{M}_E)
&\ge d_n(\mathcal{M}_E(\theta=0))\\
&= \inf_{\mathcal{V}_n}\sup_{\substack{\theta=0 \\ t\in[0,T]}} \inf_{v\in \mathcal{V}_n}  \| u(\cdot,t,\theta) - v(\cdot)\|_{L^2([-1,1]^2)}\\
&= \inf_{\mathcal{V}_n}\sup_{t\in[0,T]} \inf_{v\in \mathcal{V}_n}  \| u(\cdot,t,\theta=0) - v(\cdot)\|_{L^2([-1,1]^2)}\\
&= d_n(\{u_0(x-t,y):t\in[0,T]\}) \\
& \ge cn^{-1/2}
\end{aligned}
\]
The first inequality follows from the monotonicity of the Kolmogorov $n$-width since 
$\mathcal{M}_E(\theta=0)\subset\mathcal{M}_E$. 
The last inequality is a direct consequence of inequality~\eqref{eqn:1d_trans_ko}.

For the Lagrangian frame, we still consider the vector-valued function
\[
\begin{pmatrix}
\chi \\ \zeta \\ u
\end{pmatrix} 
=
\begin{pmatrix} \hat{x} \\ \hat{y}\\ u_0(\hat{x},\hat{y})\\ \end{pmatrix} 
+\begin{pmatrix} t\cos\theta \\ 0\\ 0\\ \end{pmatrix} 
+
\begin{pmatrix} 0 \\ t\sin\theta\\ 0\\ \end{pmatrix}
\in
\text{span}\left\{ \begin{pmatrix}\hat{x}\\ \hat{y}\\ u_0\end{pmatrix},
    \begin{pmatrix}1 \\ 0\\ 0\end{pmatrix},
    \begin{pmatrix}0 \\ 1\\ 0\end{pmatrix}
    \right\}
\]
Hence, the Kolmogorov $n$-width in the Lagrangian frame is 0 when $n\ge 3$.
\end{proof}

\subheading{1D advection-diffusion problem.} We now include the diffusion term and consider the one-dimensional advection–diffusion equation

\begin{subequations}
\label{eq:1d-n-width-eulerian-with-diffusion}
\begin{align}
&\frac{\partial u}{\partial t}+\frac{\partial u}{\partial x}=\mu \frac{\partial^2 u}{\partial x^2},
\qquad x\in\mathbb{R},\ t\ge 0, \\
&u(x,0)=(G(\cdot,\sigma_0)\ast \phi)(x)
 =\int_{\mathbb{R}} G(x-y,\sigma_0)\,\phi(y)\,dy, \\
&G(x,t)=\frac{1}{\sqrt{4\pi\mu t}}
  \exp\!\left(-\frac{x^2}{4\mu t}\right), \\
&\phi(x)=H(x-0.35)-H(x-0.65)
 =\begin{cases}
1, & 0.35<x<0.65, \\
0, & \text{otherwise}.
\end{cases}
\end{align}
\end{subequations}
where $H(x)$ denotes the Heaviside function and $G(x,t)$ is the heat kernel. Here, 
$\mu=10^{-4}$ is the diffusion coefficient.

In the Eulerian frame, the solution is given by $u(x,t)=w(x-t,t)$ where $w(x,t)=(G(\cdot,t+\sigma_0)\ast \phi)(x)$ is the analytic solution of heat equation $w_t=\mu w_{xx}$. The solution manifold is defined as

\[
\mathcal{M}_E=\{F(\frac{x-t-0.35}{\sqrt{2\mu(t+\sigma_0)}})-F(\frac{x-t-0.65}{\sqrt{2\mu(t+\sigma_0)}}): t\in[0,T]\}\subset L^2(\mathbb{R})
\]
where $t$ is the parameter and $F(x)=\frac{1}{\sqrt{2\pi}}\int_{-\infty}^x\exp(-s^2/2)ds$ is the cumulative distribution function. 
The solution manifold is still a family of translated window functions, but each solution is smoothed by a Gaussian kernel due to diffusion. Moreover, each solution belongs to $C^\infty(\mathbb{R})$. The frontier propagates with constant velocity, while the front thickness $d$ grows slowly as $d \sim \sqrt{\mu t}$. For small times $t$ and diffusion coefficient $\mu$, the singular-value decay of the solution matrix remains comparable to that of the pure translation problem.

In the Lagrangian frame, the governing equation becomes:
\[
\left\{\begin{aligned}
&\frac{d\chi}{dt} = 1\\
&\frac{du}{dt}|_{x=\chi} = \mu u_{xx}|_{x=\chi} \\
\end{aligned}\right.
\implies
\left\{\begin{aligned}
&\chi(\hat{x},t) = \hat{x}+t \\
& u(\chi,t)= (G(\cdot,t+\sigma_0)\ast \phi)(\hat{x})
\end{aligned}\right.
\]
The corresponding Lagrangian solution manifold is
\[
\mathcal{M}_L =
\left\{
\begin{pmatrix}
\chi \\ u
\end{pmatrix} 
:\chi=\hat{x}+t,u=F(\frac{\hat{x}-0.35}{\sqrt{2\mu(t+\sigma_0)}})-F(\frac{\hat{x}-0.65}{\sqrt{2\mu(t+\sigma_0)}})\right\}
\]
where $t\in[0,T]$ is the varying parameter.

Along the characteristic line $\chi$, the advection-diffusion problem reduces to a pure diffusion problem. Therefore, the approximation properties of $\mathcal{M}_L$ are governed primarily by diffusion effects. Consequently, the Kolmogorov $n$-width $d_n(\mathcal{M}_L)$ exhibits an exponential decay rate comparable to the decay rate of a heat equation.

\begin{theorem}
    In the Eulerian frame, for sufficiently small $T,\mu>0$, there exists a constant $c_2>0$ and $c>0$ such that for all $n\in\mathbb{N}^*$,
    \begin{align}
        & d_n(\mathcal{M}_E)\geq c_2 n^{-1/2} - \sqrt{c}(\mu(T+\sigma_0))^{1/4}
    \end{align}
    
In the Lagrangian frame, for all $n>2$,
    \[
    d_n(\mathcal M_L) \le C\frac{q^{n-2}}{1-q},\;\text{where}\;q=\frac{T/2}{T/2+\sigma_0}.
    \]
\end{theorem}
\begin{proof}
    \subheading{In the Eulerian frame}, consider an auxiliary solution manifold 
    \[
    M^{\textrm{ref}}_E:=\{\phi(x-t),t\in[0,T]\}
    \]
    From equation~\eqref{eq:1d-n-width-eulerian-with-diffusion}, recall that $\phi(x)=H(x-0.35)-H(x-0.65)$, where $H(x)$ is the Heaviside function. For any $t'\in [0,T], u(x,t')\in\mathcal{M}_E$, we have
    \begin{align*}
    & \|u(x,t')-\phi(x-t')\|^2 =\|(G(\cdot,t'+\sigma_0)\ast\phi)(x)-\phi(x-t')\|^2 \\
    =& \| [F(\frac{\hat{x}-t'-0.35}{\sqrt{2\mu(t'+\sigma_0)}})-F(\frac{\hat{x}-t'-0.65}{\sqrt{2\mu(t'+\sigma_0)}})] - [H(\hat{x}-t'-0.35)-H(\hat{x}-t'-0.65)]\|^2 \\
    \le& 2 \| F(\frac{\hat{x}}{\sqrt{2\mu(t'+\sigma_0)}}) - H(\hat{x}) \|^2=2\int_{-\infty}^\infty [F(\frac{\hat{x}}{\sqrt{2\mu (t'+\sigma_0)}}) - H(\hat{x})]^2d\hat{x} \\
    =& 2\sqrt{2\mu (t'+\sigma_0)}\int_{-\infty}^\infty [F(s) - H(\sqrt{2\mu (t'+\sigma_0)}s)]^2ds.
    \end{align*}
    Since the Heaviside function satisfies that $H(s)=H(as)$ for any $a>0$, we have
    \[
    \begin{aligned}
    &\|u(x,t')-\phi(x-t')\|^2 \\
    \leq&2\sqrt{2\mu (t'+\sigma_0)} \int_{-\infty}^\infty [F(s) - H(s)]^2ds \\
    =&2\sqrt{2\mu (t'+\sigma_0)} \cdot \frac{\sqrt{2}-1}{\sqrt{\pi}} =\frac{4-2\sqrt{2}}{\sqrt{\pi}} \sqrt{\mu (t'+\sigma_0)}
    \le c \sqrt{\mu(T+\sigma_0)}.
    \end{aligned}
    \] where the constant $c=\frac{4-2\sqrt{2}}{\sqrt{\pi}}$. 
    As a result,
    \[\max\{\sup_{u^{\textrm{ref}}\in\mathcal{M}_E^{\textrm{ref}}}\inf_{u\in\mathcal{M}_E}\| u^{\textrm{ref}}-u\|,\sup_{u\in\mathcal{M}_E}\inf_{u^{\textrm{ref}}\in\mathcal{M}^{\textrm{ref}}_E}\| u^{\textrm{ref}}-u\|\} \le \sqrt{c}(\mu(T+\sigma_0))^{1/4}.
    \]
    
    For any $u\in\mathcal{M}_E$ and $n$-dimensional linear subspace $\mathcal{V}_n$, we have
    \[\begin{aligned}
    &\inf_{v\in\mathcal{V}_n}\|u-v\| \le \|u-u^{\textrm{ref}}\|+\inf_{v\in\mathcal{V}_n}\|u^{\textrm{ref}}-v\|, \quad \forall v\in\mathcal{V}_n \\
    \implies 
    & \sup_{u\in\mathcal{M}_E}\inf_{v\in\mathcal{V}_n}\|u-v\| \le \sqrt{c}(\mu (T+\sigma_0))^{1/4} +\sup_{u^{\textrm{ref}}\in\mathcal{M}_E^{\textrm{ref}}}\inf_{v\in\mathcal{V}_n}\|u^{\textrm{ref}}-v\|\\
    \implies & \inf_{\mathcal{V}_n}\sup_{u\in\mathcal{M}_E}\inf_{v\in\mathcal{V}_n}\|u-v\| \le \sqrt{c}(\mu (T+\sigma_0))^{1/4}+\inf_{\mathcal{V}_n}\sup_{u^{\textrm{ref}}\in\mathcal{M}_E^{\textrm{ref}}}\inf_{v\in\mathcal{V}_n}\|u^{\textrm{ref}}-v\|
    \end{aligned}\]
    As a consequence,
    \[
    d_n(\mathcal{M}_E)\le \sqrt{c}(\mu (T+\sigma_0))^{1/4} + d_n(\mathcal{M}_E^{\textrm{ref}}).
    \] Similarly, we can prove that $ d_n(\mathcal{M}_E^{\textrm{ref}})\le \sqrt{c}(\mu (T+\sigma_0))^{1/4} + d_n(\mathcal{M}_E)$. Hence,
    \[
    |d_n(\mathcal{M}_E)-d_n(\mathcal{M}_E^{\textrm{ref}})|\le \sqrt{c}(\mu(T+\sigma_0))^{1/4}
    \]

    Thus, we have
    \begin{equation}
        d_n(\mathcal{M}_E)\geq d_n(\mathcal{M}_E^{\textrm{ref}})- \sqrt{c}(\mu (T+\sigma_0))^{1/4}=c_2 n^{-1/2} - \sqrt{c}(\mu(T+\sigma_0))^{1/4}.
    \end{equation}
where the estimate $d_n(\mathcal{M}_E^{\textrm{ref}})\geq c_2 n^{-1/2}$ follows from equation~\eqref{eq:1d-adv-dn}.
    
    \subheading{In the Lagrangian frame}, the state variable satisfies $u(\hat x,t):=w(\hat x,t)$ where $w(\cdot,t):=(G(\cdot,t+\sigma_0)*\phi)(\cdot)$ is the analytic solution of heat equation. Since trajectory component $\chi(\cdot,t)\in\mathrm{span}\{1,\hat x\}$, it contributes at most $2$ dimensions. Therefore, for $n>2$,
    \[
    d_n(\mathcal M_L) = d_{n-2}(\mathcal{M}_L^w),
    \text{ where }
    \mathcal{M}_L^w:=\{w(\cdot,t):t\in[0,T]\}\subset L^2(\mathbb{R}).
    \]
    So, we only need to estimate the $n$-width for the $u$-component.
    
    \subheading{Step 1: Fourier representation.}
    Let $\mathcal F:L^2(\mathbb R_x)\to L^2(\mathbb R_\xi)$ be the unitary  Fourier transform $\|f\|_{L^2_x}=\|\widehat f\|_{L^2_\xi}$.
    Consequently, for any compact set $K\subset L^2$,
    \[
    d_n(K)=d_n(\mathcal F K),
    \]
    since $\mathcal F$ preserves all distances $\|\mathcal{F}(f)-\mathcal{F}(g)\|_{L^2_\xi} = \|f-g\|_{L^2_x}$. Hence it suffices to bound the $n$-width in the Fourier domain
    \[
    \widehat{\mathcal M}^w_L:=\{\widehat w(\cdot,t):t\in[0,T]\}\subset L^2(\mathbb R_\xi).
    \]
    The Fourier transform of  $w(\cdot,t)$ is
    \[
    \widehat w(\xi,t)=e^{-\mu (t+\sigma_0)\xi^2}\,\widehat\phi(\xi).
    \]
    \subheading{Step 2: approximation error of Taylor polynomials.}
    Fix $t^*:=T/2$. We show that, for every $\xi\in\mathbb{R}$, the Taylor series of $\widehat{w}(\xi,t)$ about $t=t^*$ converges:
\[
\widehat{w}(\xi,t)
=
\sum_{k=0}^\infty \frac{\partial_t^k \widehat{w}(\xi,t^*)}{k!}(t-t^*)^k
=
\sum_{k=0}^\infty
\frac{(-\mu\xi^2)^k}{k!}
e^{-\mu(t^*+\sigma_0)\xi^2}\widehat{\phi}(\xi)\,(t-t^*)^k.
\]
Define $y:=\mu(t^*+\sigma_0)\xi^2.$ Then,
\[
\frac{(\mu\xi^2)^k}{k!}e^{-\mu(t^*+\sigma_0)\xi^2}
=
\frac{1}{(t^*+\sigma_0)^k}\frac{y^k e^{-y}}{k!}.
\]
The function $y\mapsto y^k e^{-y},\;y\in[0,\infty)$ attains its maximum  at $y=k$, so
\[
\frac{y^k e^{-y}}{k!}
\le
\frac{k^k e^{-k}}{k!}.
\]
Now, by the Stirling lower bound \cite{robbins1955remark},
\[
k!\ge \sqrt{2\pi k}\left(\frac{k}{e}\right)^k,
\qquad k\ge 1,
\]
it follows that
\[
\frac{k^k e^{-k}}{k!}
\le
\frac{1}{\sqrt{2\pi k}}
\le 1
\Rightarrow
\frac{(\mu\xi^2)^k}{k!}e^{-\mu(t^*+\sigma_0)\xi^2}
\le
\frac{1}{(t^*+\sigma_0)^k}.
\]
In addition, using $t^*=\frac{T}{2}$ and $t\in[0,T]$, we have
\[
\left\|
\frac{\partial_t^k\widehat{w}(\xi,t^*)}{k!}(t-t^*)^k
\right\|_{L^2_\xi}
\le
\frac{\|\widehat{\phi}\|_{L^2_\xi}}{(t^*+\sigma_0)^k}|t-t^*|^k
\le
\frac{(T/2)^k}{(T/2+\sigma_0)^k}\|\phi\|_{L^2}. 
\]
Hence,
\[\left\|\sum_{k=0}^\infty\frac{\partial_t^k\widehat{w}(\xi,t^*)}{k!}(t-t^*)^k\right\|_{L^2_\xi} \leq \|\phi\|_{L^2}\sum_{k=0}^\infty(\frac{T/2}{T/2+\sigma_0})^k=\|\phi\|_{L^2}\frac{T/2+\sigma_0}{\sigma_0}<\infty.
\]
As a result, the series
$\widehat{w}(\xi,t)
=
\sum_{k=0}^\infty\frac{\partial_t^k\widehat{w}(\xi,t^*)}{k!}(t-t^*)^k
$
converges absolutely in $L^2_\xi$, uniformly for $t\in[0,T]$. Therefore, the Taylor polynomial of $n-1$ degree
$p_{n-1}(\cdot,t):=
\sum_{k=0}^{n-1}
\frac{\partial_t^k\widehat{w}(\xi,t^*)}{k!}(t-t^*)^k$
satisfies
\[
\begin{aligned}
\|\widehat w(\cdot,t)-p_{n-1}(\cdot,t)\|_{L^2_\xi} 
\le
\sum_{k=n}^\infty
\left\|
\frac{\partial_t^k\widehat{w}(\xi,t^*)}{k!}(t-t^*)^k
\right\|_{L^2_\xi} 
\le
\sum_{k=n}^\infty \|\phi\|_{L^2} q^k
=
\frac{\|\phi\|_{L^2}}{1-q}\,q^n,
\end{aligned}
\]
where $q=(T/2)/(T/2+\sigma_0)\in(0,1).$

\subheading{Step 3: from approximation error to an upper bound of the $n$-width.} Define the $n$-dimensional subspace
    \[
    W_n:=\mathrm{span}\{\widehat{w}(\xi,t^*),\partial_t\widehat{w}(\xi,t^*),\frac{\partial^2_t\widehat{w}(\xi,t^*)}{2},\cdots,\frac{\partial^{n-1}_t\widehat{w}(\xi,t^*)}{(n-1)!}\}\subset L^2(\mathbb R_\xi).
    \]
    then $p_{n-1}(t)\in W_n$ for every $t$. Hence
    \[
    \inf_{v\in W_n}\|\widehat w(\cdot,t)-v\|_{L^2_\xi}
    \le \|\widehat w(\cdot,t)-p_{n-1}(t)\|_{L^2_\xi}.
    \]
    The definition of $d_n$ gives
    \[\begin{aligned}
    d_n(\widehat{\mathcal M}^w_L)
    & =\inf_{\dim(V)=n}\sup_{t\in[0,T]}\inf_{v\in V}\|\widehat w(\cdot,t)-v\|_{L^2_\xi} \\
    & \le \sup_{t\in[0,T]}\inf_{v\in W_n}\|\widehat w(\cdot,t)-v\|_{L^2_\xi}
    \le \frac{\|\phi\|_{L^2}}{1-q}\,q^n.
    \end{aligned}\]
    By the $L^2$-isometry of $\mathcal F$, this implies the same bound for $d_n(\mathcal M^w_L)$. Thus 
    \[
    d_n(\mathcal{M}_L)=d_{n-2}(\mathcal{M}^w_L) \le \frac{\|\phi\|_{L^2}}{1-q}\,q^{n-2},\;0<q=\frac{T/2}{T/2+\sigma_0}<1.
    \]
\end{proof}

\subsection{Strong correlation in the Lagrangian frame}\label{subsec:lag_interp}
In transport-dominated dynamics, Eulerian snapshots primarily differ by spatial shifts. As these shifts increase, the overlap between snapshots diminishes, leading to weak correlation between training and future solutions. Consequently, prediction in the Eulerian frame becomes an extrapolation toward uncorrelated, out-of-distribution states. In contrast, the Lagrangian frame follows characteristic trajectories and encodes shifts into the coordinate variable, effectively aligning future and historical solutions. As a result, the remaining temporal variation reflects mainly changes in shape or amplitude rather than position, which preserves strong correlation across time. Hence, prediction in the Lagrangian frame is closer to interpolation between correlated states than to extrapolation toward uncorrelated out-of-distribution states.

We quantify this effect using a coherence coefficient, defined as the maximal normalized correlation between a snapshot at time $t$ and the training set:
\begin{equation}\label{eq:coherence}
\gamma_\star(t)
=
\max_{s\in[0,T_{\text{train}}]}
\frac{\big|\langle w_\star(\cdot,t),\,w_\star(\cdot,s)\rangle\big|}
{\|w_\star(\cdot,t)\|\,\|w_\star(\cdot,s)\|},\qquad \star\in\{E,L\},
\end{equation}
where $\langle\cdot,\cdot\rangle$ and $\|\cdot\|$ denote the spatial $L^2$ inner product and norm (implemented as discrete dot products on the fixed grid). Here, $w_E(\cdot,t):=u(\cdot,t)$ is the Eulerian snapshot, while $w_L(\cdot,t)=(\chi(\cdot,t),u(\cdot,t))$ denotes the corresponding Lagrangian representation.

\begin{figure}[htbp]
    \centering
    \includegraphics[width=0.5\textwidth]{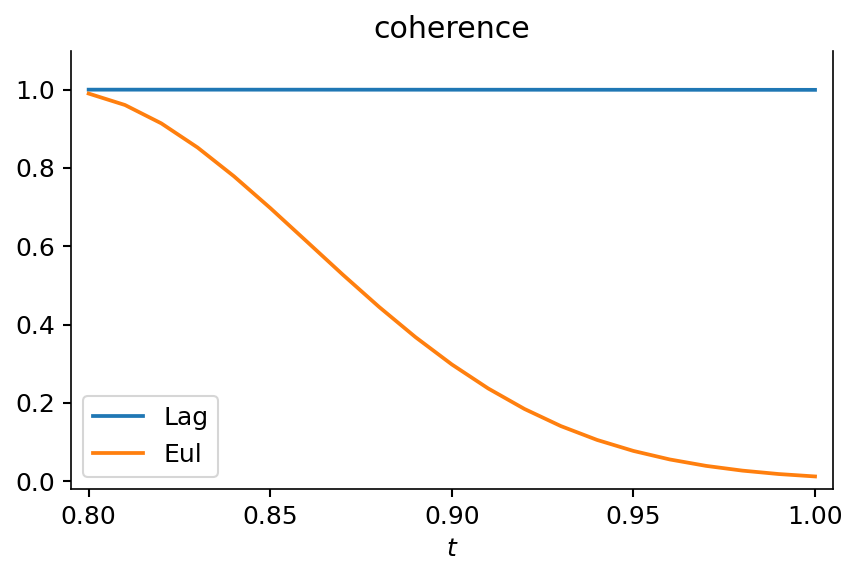}
    \caption{Coherence in $(T_{\text{train}},T]$. The Eulerian-frame coherence decays rapidly in the forecast, indicating weak correlation with the training set (extrapolation), whereas the Lagrangian coherence remains high due to alignment (closer to interpolation).}
    \label{fig:1d_adv_coherence}
\end{figure}

Figure~\ref{fig:1d_adv_coherence} reports $\gamma_E(t)$ and $\gamma_L(t)$ on the prediction window $t\in(T_{\text{train}},T]$ (with $T_{\text{train}}=0.8$ and $T=1$). The coherence in the Eulerian frame decays rapidly toward zero, indicating that future snapshots become nearly uncorrelated with any training snapshot. Hence, prediction in the Eulerian frame is effectively an out-of-distribution extrapolation problem. In contrast, the Lagrangian coherence remains high, implying that future states stay close to structures already observed during training.

Therefore, for transport-dominated systems, prediction in the Lagrangian frame is closer to an interpolation between correlated states rather than an extrapolation towards uncorrelated out-of-distribution states. 
\section{Lagrangian ROMs}\label{sec:lagrom}
As discussed in Section~\ref{sec:motivation}, representing data for transport-dominated dynamics in the Lagrangian frame can potentially improve both compression ability and future prediction performance, since the Kolmogorov $n$-width decays faster and the solutions exhibit stronger correlations in it. POD and DMD have been extended to the Lagrangian frame in Lagrangian POD~\cite{Mojgani2017-lagpod} and Lagrangian DMD~\cite{Lu2020-lagdmd,lu2021dynamic}, respectively. To the best of our knowledge, an autoencoder-based ROM in the Lagrangian frame has not been developed. In this paper, we extend the Eulerian non-intrusive autoencoder ROM based on parametric DMD  \cite{Duan2024-cae-hodmd} to the Lagrangian frame. A byproduct of this development is a Lagrangian parametric DMD, which can be seen as a Lagrangian extension of parametric DMD in \cite{Andreuzzi2023-pdmd-partition-stack} or a parametric extension of the Lagrangian DMD in \cite{Lu2020-lagdmd,lu2021dynamic}. 

Specifically, the Lagrangian autoencoder ROM learns a nonlinear latent manifold from coherent augmented snapshots in the Lagrangian frame offline. Fast online future prediction for unseen parameters can be performed using interpolation-based parametric DMD \cite{Andreuzzi2023-pdmd-partition-stack}.

In this section, we first outline the Lagrangian representation of solution data in Section~\ref{subsec:lag_data}, 
and then introduce Lagrangian autoencoder and Lagrangian parametric DMD in Section~\ref{subsec:lag_ae_rom} and Section~\ref{subsec:lag_pod_rom}.

\subsection{Lagrangian data}\label{subsec:lag_data}
We first describe the construction of Lagrangian snapshots. For each parameter $\mu\in\mathcal{P}\subset\mathbb{R}^{d_\mu}$, the characteristic line $\chi(\hat{x},t;\mu)$ transports a reference coordinate $\hat{x}\in\Omega$ to  a new location $\chi$ at time $t$. We denote the Lagrangian coordinate and solution as $\chi(\hat{x},t;\mu)$ and $u\bigl(\chi(\hat{x},t;\mu),t;\mu\bigr).$
Let $\{\mu_i\}_{i=1}^{n_p}\subset\mathcal{P}$ and $\{t_k\}_{k=0}^{n_t}\subset[0,T]$ denote the training parameter samples and time instances. For each $(\mu_i,t_k)$, we form an augmented Lagrangian snapshot
\begin{equation}\label{eq:lag_snapshot_stack}
m^{\mu_i,t_k} = \mathrm{stack}\!\left(\chi^{\mu_i,t_k},\,u^{\mu_i,t_k}\right),
\end{equation}
where $\chi^{\mu_i,t_k}$ and $u^{\mu_i,t_k}$ denote the Lagrangian coordinates and solutions. The stacking operator is method-dependent and will be specified for each Lagrangian ROM. In contrast, a conventional Eulerian representation stores only $u(x,t_k;\mu_i)$ on a fixed physical grid.

\subsection{Lagrangian nonlinear ROM}\label{subsec:lag_ae_rom}
In this section, we present the Lagrangian convolutional autoencoder (LagCAE), a novel architecture specifically designed to exploit Lagrangian snapshots. This method can be viewed as an extension of classical autoencoder-based ROM in \cite{Lee2020-cae-rom,Duan2024-cae-hodmd} from the Eulerian frame to the Lagrangian frame to better leverage coherent structures of transport-dominated systems encoded in the Lagrangian frame.

\subsubsection{Offline training based on LagCAE}\label{subsubsec:lagcae_offline}

\subheading{Data treatment.} For clarity, we describe the one-dimensional case. We denote the discrete Lagrangian coordinate and state vectors by
\[
\chi^{\mu_i,t_k}=(\chi^{\mu_i,t_k}_1,\dots,\chi^{\mu_i,t_k}_{n_s})\in\mathbb{R}^{n_s},\qquad
u^{\mu_i,t_k}=(u^{\mu_i,t_k}_1,\dots,u^{\mu_i,t_k}_{n_s})\in\mathbb{R}^{n_s}.
\]
We assemble them into a two-channel tensor
\[
m^{\mu_i,t_k} = \mathrm{stack}\bigl(\chi^{\mu_i,t_k},u^{\mu_i,t_k}\bigr)
=
\left[\begin{array}{c}
    \rule[.5ex]{1em}{0.4pt}\; \chi^{\mu_i,t_k} \;\rule[.5ex]{1em}{0.4pt} \\
    \rule[.5ex]{1em}{0.4pt}\; u^{\mu_i,t_k} \;\rule[.5ex]{1em}{0.4pt} 
\end{array}\right]
\in \mathbb{R}^{2\times n_s}.
\]
where $\mathrm{stack}(\cdot,\cdot)$ stacks the two arrays along a channel dimension. This representation preserves the coupling between trajectories and solution values, ensuring that compression and reconstruction treat them consistently.

\subheading{Architecture.} We adopt a convolutional autoencoder architecture to learn a low-dimensional nonlinear manifold for augmented Lagrangian data. 
The encoder--decoder pair reads
\[
\mathcal{G}_e(\cdot;\theta_{\mathrm{enc}}):\mathbb{R}^{2\times n_s}\to\mathbb{R}^{r},\qquad
\mathcal{G}_d(\cdot;\theta_{\mathrm{dec}}):\mathbb{R}^{r}\to\mathbb{R}^{2\times n_s},
\]
where the encoder $\mathcal{G}_e$ and decoder $\mathcal{G}_d$ are neural networks with trainable parameters $\theta_{\mathrm{enc}}$ and $\theta_{\mathrm{dec}}$, respectively. Given a snapshot $m^{\mu_i,t_k}$, the encoder $\mathcal{G}_{e}$ compresses it to a latent coordinate
$h^{\mu_i,t_k}=\mathcal{G}_e(m^{\mu_i,t_k};\theta_{\mathrm{enc}})$, while the decoder $\mathcal{G}_d$ reconstructs it from the latent coordinate as
$\widehat{m}^{\mu_i,t_k}=\mathcal{G}_d(h^{\mu_i,t_k};\theta_{\mathrm{dec}})$.

\subheading{Loss.} We train LagCAE by minimizing a combination of value reconstruction and a gradient-based loss term:
\begin{align}
\mathcal{L}_{\mathrm{recon}}
&=\frac{1}{n_pn_t}\sum_{i,k}
\left\|\widehat{m}^{\mu_i,t_k}-m^{\mu_i,t_k}\right\|_2^2,\\
\mathcal{L}_{\mathrm{grad}}
&=\frac{1}{n_pn_t}\sum_{i,k}
\left\|\nabla_{\hat{x}}\widehat{m}^{\mu_i,t_k}-\nabla_{\hat{x}}m^{\mu_i,t_k}\right\|_2^2,
\end{align}
where $\nabla_{\hat{x}}$ denotes discrete differentiation along the reference coordinate $\hat{x}$, implemented as finite differences along the spatial index (e.g., \texttt{torch.gradient}). We note that the differentiation is performed with respect to the uniform reference coordinates $\hat{x}$ rather than on the nonuniform physical grid.

The overall objective is
\begin{equation}\label{eq:loss-lagcae}
\mathcal{L}=\mathcal{L}_{\mathrm{recon}}+\lambda_{\mathrm{grad}}\mathcal{L}_{\mathrm{grad}},\qquad \lambda_{\mathrm{grad}}\ge 0.
\end{equation}
Similar Sobolev-type training objectives have been shown to improve reconstruction quality and generalization in image reconstruction and surrogate modeling \cite{Czarnecki2017-sobolev-training,Deshpande2017-ae-image-gradientloss}. 

\subsubsection{Online prediction via pDMD}\label{subsubsec:pdmd_online}
For online prediction, we adopt a non-intrusive strategy via interpolation-based parametric DMD (pDMD) proposed in \cite{Andreuzzi2023-pdmd-partition-stack,Duan2024-cae-hodmd}. Specifically, pDMD is applied to learn the dynamics of latent coordinates in the reduced latent space. Other parametric extensions of DMD, e.g., \cite{Huhn2023-pdmd-eigen-or-koopman,Song2024-pdmd-winn}, may be leveraged.

\begin{remark}
In addition to pDMD, other online strategies could be considered. For example, we can apply the Galerkin and Petrov-Galerkin projections for autoencoder-based ROM in \cite{Lee2020-cae-rom} or alternative non-intrusive online approaches such as neural-network-based interpolation \cite{fresca2021comprehensive}, Gaussian regression \cite{Bai2021-gaussian-reg,Ortali2021-gaussian-reg} and greedy latent space sparse dynamics identification (gLaSDI) \cite{He2023a-gladi}.
\end{remark}

\subheading{Offline learning of latent dynamics through DMD.} After training, all augmented snapshots are encoded to latent coordinates. 
For each training parameter $\mu_i$, we collect the latent trajectory
\[
H^{\mu_i}=\bigl[h^{\mu_i,t_0},\dots,h^{\mu_i,t_{n_t}}\bigr]\in\mathbb{R}^{r\times (n_t+1)}.
\]
Then, for each parameter, we construct a DMD-based ROM and represent the associated latent dynamics as $\tilde{A}_{\mu_i}\in\mathbb{R}^{r\times r}$ such that $h^{\mu_i,t_{k+1}}\approx \tilde{A}_{\mu_i}h^{\mu_i,t_k}$. Given an initial latent state $h^{\mu_i,t_n}$ corresponding to a parameter $\mu_i$ in the training set, the ROM predicts future latent states as
\[
h^{\mu_i,t_{n+k}}=(\tilde{A}_{\mu_i})^k h^{\mu_i,t_n}.
\]

\subheading{Online prediction through parametric interpolation.} To enable future prediction for an unseen parameter $\mu^*$ and future time $t_{n+k}$, we perform a parametric interpolation in the latent space. First, we use the DMD models constructed for the training parameters to predict latent coordinates $h^{\mu_i,t_{n+k}}$ for each $\mu_i$. Then, we build an interpolant $\mathcal{I}_k(\cdot)$ such that
$\mathcal{I}_k(\mu_i)\approx h^{\mu_i,t_{n+k}}$, and define
\[
h^{\mu^*,t_{n+k}}:=\mathcal{I}_k(\mu^*).
\]
Finally, the predicted future solution in the Lagrangian frame can be obtained through decoding
\[
\widehat{m}^{\mu^*,t_{n+k}}=\mathcal{G}_d\!\left(h^{\mu^*,t_{n+k}};\theta_\mathrm{dec}\right)
=\mathrm{stack}\!\left(\widehat{\chi}^{\mu^*,t_{n+k}},\,\widehat{u}^{\mu^*,t_{n+k}}\right).
\]
In our experiments, $\mathcal{I}_k$ is taken to be radial basis function interpolation. Other interpolation methods (e.g., Gaussian-process regression) may also be used.

\subsubsection{Non-intrusive data-driven ROM}
The overall algorithm of Lagrangian CAE with parametric DMD (LagCAE--pDMD) is summarized in Algorithms~\ref{alg:lagcae_offline}--\ref{alg:lagcae_online}, which describe the offline and online stages, respectively. The most computationally expensive steps are performed in the offline stage, including solving the full-order model in the Lagrangian frame and training the LagCAE. During the online stage, efficient predictions for future times and new parameters can be made through DMD-based latent dynamic evolution and interpolation in the reduced space, followed by a single decoder application. 

\subheading{Improved prediction performance.}
We now assess the predictive performance of the proposed Lagrangian ROM. Returning to the 1D advection problem in Section~\ref{sec:failure}, Figure~\ref{fig:1d_adv_lag} shows that the Lagrangian ROM accurately predicts the future solution, whereas the Eulerian ROM fails to extrapolate and exhibits spurious oscillations and incorrect wavefront locations.

In addition, the latent dynamics is more stable in the Lagrangian frame. As shown in Figure~\ref{fig:1d_adv_latent}, the latent coordinates in the Lagrangian frame evolve smoothly and remain bounded throughout the prediction interval. In contrast, the Eulerian latent trajectories exhibit significant drift, particularly during extrapolation from $t=0.8$ to $t=1.0$. 

\begin{algorithm}[H]
    \caption{LagCAE--pDMD: Offline training.}\label{alg:lagcae_offline}
    \KwIn{Lagrangian snapshots of full-order solution: $\{m^{\mu_i,t_k}\}_{\mu_i,t_k}$}
    \KwOut{Autoencoder $(\mathcal{G}_e,\mathcal{G}_d)$; latent DMD operators $\{\tilde{A}_{\mu_i}\}_{\mu_i}$}
    \begin{enumerate}
            \item Train autoencoder $\mathcal{G}=\mathcal{G}_d\circ\mathcal{G}_e$ by minimizing~\eqref{eq:loss-lagcae} to obtain $\theta_{\mathrm{enc/dec}}$.
            \item Encode all snapshots to obtain latent trajectories $H^{\mu_i} =[\mathcal{G}_e(m^{\mu_i,t_0}),\dots,\mathcal{G}_e(m^{\mu_i,t_{n_t}})]$ for all $\mu_i$.
            \item Construct DMD operator $\tilde{A}_{\mu_i}$ for each parameter based on latent coordinates for each $\mu_i$.
    \end{enumerate}
\end{algorithm}

\begin{algorithm}[H]
    \caption{LagCAE--pDMD: Online prediction.}
    \label{alg:lagcae_online}
    \KwIn{Prediction step $k$ and new parameter $\mu^*$}
    \KwOut{Eulerian solution $\widehat{u}_{E}^{\mu^*,t_{n+k}}$}
    \begin{enumerate}
        \item For each training parameter $\mu_i$, evolve the latent dynamics: $h^{\mu_i,t_{n+k}}=(\tilde{A}_{\mu_i})^k h^{\mu_i,t_n}.$

        \item Interpolate across the parameter space to obtain $h^{\mu^*,t_{n+k}}=\mathcal{I}_k(\mu^*)$
        from the latent states $\{h^{\mu_i,t_{n+k}}\}_{i=1}^{n_p}$.
        \item Decode the interpolated latent state: $\mathrm{stack}\!\left(\widehat{\chi}^{\mu^*,t_{n+k}},\,\widehat{u}^{\mu^*,t_{n+k}}\right)=\mathcal{G}_d\!\left(h^{\mu^*,t_{n+k}}\right).$
        \item Reconstruct the Eulerian field: $\widehat{u}_{E}^{\mu^*,t_{n+k}}=\mathrm{Interp}\!\left(\widehat{\chi}^{\mu^*,t_{n+k}},\,\widehat{u}^{\mu^*,t_{n+k}}\right).$
    \end{enumerate}
\end{algorithm}

\begin{figure}[htbp]
    \centering
    \begin{subfigure}[b]{0.4\textwidth}
        \includegraphics[width=\textwidth]{fig/1d_adv_cae8.png}
        \label{subfig:1d_adv_eul_cae_2}
        \caption{Eulerian ROM}
    \end{subfigure}\hfill
    \begin{subfigure}[b]{0.4\textwidth}
        \includegraphics[width=\textwidth]{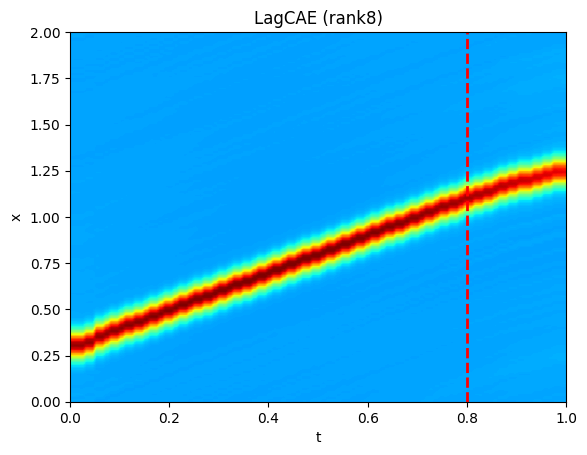}
        \label{subfig:1d_adv_lag_cae}
        \caption[Lagrangian ROM]{Lagrangian ROM}
    \end{subfigure}
    \caption{Prediction results of nonlinear ROMs. The red dashed line marks $t=T_{\text{train}}=0.8$, separating reconstruction (left) from prediction (right). The Lagrangian ROM preserves the transported structure during prediction, while the Eulerian ROM exhibits spurious artifacts after $t=0.8$.}
    \label{fig:1d_adv_lag}
\end{figure}

\begin{figure}[htbp]
    \centering
    \begin{subfigure}[b]{0.33\textwidth}
        \includegraphics[width=\textwidth]{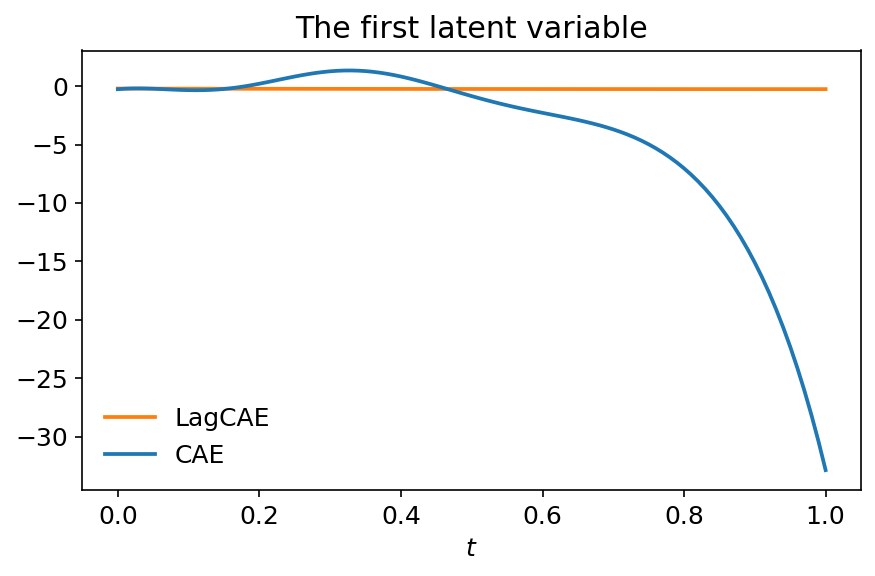}
        \label{subfig:1d_adv_latent1}
    \end{subfigure}\hfill
    \begin{subfigure}[b]{0.33\textwidth}
    \includegraphics[width=\textwidth]{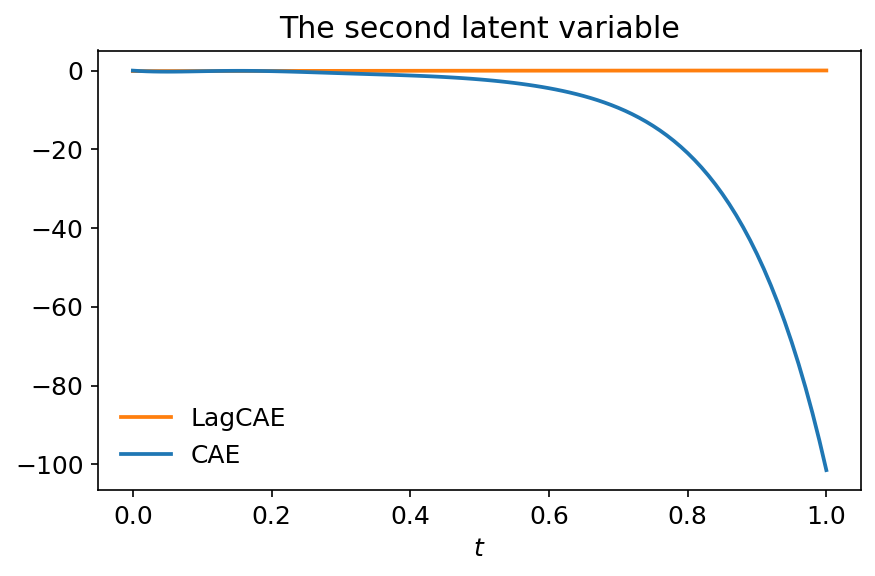}
        \label{subfig:1d_adv_latent2}
    \end{subfigure}
    \begin{subfigure}[b]{0.33\textwidth}
    \includegraphics[width=\textwidth]{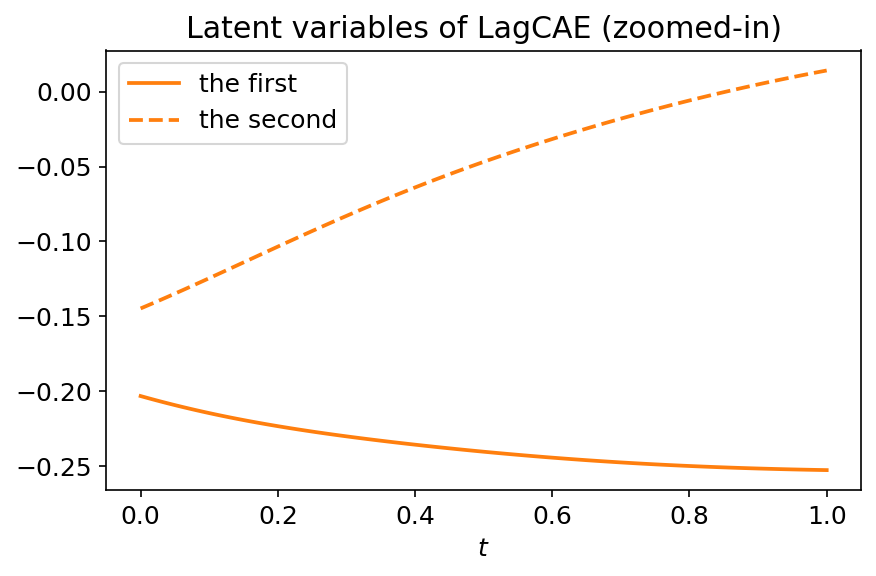}
        \label{subfig:1d_adv_lag_latent}
    \end{subfigure}
    \caption{Latent trajectories during prediction for the 1D advection test. The Lagrangian latent trajectories remain smooth and bounded, while the Eulerian latent trajectories exhibit drift during prediction.}
    \label{fig:1d_adv_latent}
\end{figure}
\subsection{Lagrangian parametric DMD}\label{subsec:lag_pod_rom}
The online stage of the proposed non-intrusive Lagrangian autoencoder-based ROM naturally expanded as a parametric version of Lagrangian DMD by replacing the autoencoder with linear compression and lifting operators. We call the resulting method Lagrangian parametric DMD (Lag--pDMD).

This method can be seen as the linear counterpart of the proposed Lagrangian autoencoder, the Lagrangian counterpart of parametric DMD in \cite{Andreuzzi2023-pdmd-partition-stack}, and the parametric extension of the Lagrangian DMD in \cite{Lu2020-lagdmd}.  Its crucial components are as follows.

\subheading{Data treatment.} For each $(\mu_i,t_k)$, let
\[
m^{\mu_i,t_k} = \mathrm{stack}\bigl(\chi^{\mu_i,t_k},u^{\mu_i,t_k}\bigr)
=\begin{bmatrix}
        | \\
        \chi^{\mu_i,t_k}\\
        | \\
        u^{\mu_i,t_k} \\
        | \\
    \end{bmatrix}
\in\mathbb{R}^{2n_s},
\quad k=0,\dots,n_t,
\]
where the operator $\mathrm{stack}(\cdot,\cdot)$ concatenates the coordinate and solution vectors (both indexed by the reference coordinate $\{\hat{x}_j\}_{j=1}^{n_s}$) into a column vector of length $2n_s$. For each parameter $\mu_i$, the corresponding snapshots are collected:
\[
M^{\mu_i} = [m^{\mu_i,t_0},\dots,m^{\mu_i,t_{n}}]\in\mathbb{R}^{2n_s\times (n_t+1)}.
\]
All snapshots are then assembled into the global Lagrangian snapshot matrix
\[
\mathbf{M}
=
\bigl[ M^{\mu_1}, \dots, M^{\mu_{n_p}} \bigr]
\in \mathbb{R}^{2n_s \times n_p(n_t+1)}
\]
where $n_p(n_t+1)$ is the total number of snapshots.

\subheading{Offline training.} We compute a truncated singular value decomposition
\[
\mathbf{M} \approx \Phi \Sigma V^\top,
\]
and retain the first $r$ left singular vectors $\Phi_r\in\mathbb{R}^{2n_s\times r}$ as a global orthonormal basis for the augmented Lagrangian snapshots. For each training parameter $\mu_i$, we project the corresponding snapshots onto this basis:
\[
H^{\mu_i} = \Phi_r^\top M^{\mu_i} 
= \bigl[ h^{\mu_i,t_0},\dots,h^{\mu_i,t_{n_t}} \bigr] \in \mathbb{R}^{r\times (n_t+1)}.
\]
A reduced DMD operator $\tilde{A}_{\mu_i}\in\mathbb{R}^{r\times r}$ is constructed for each $\mu_i$ using the standard DMD procedure in Section~\ref{subsubsec:dmd}, applied to the shifted data pair $(H^{\mu_i,-},H^{\mu_i,+})$. 

\subheading{Online prediction}
The online stage of Lag--pDMD follows the same latent-state-interpolation pDMD strategy as in Section~\ref{subsubsec:pdmd_online} and mirrors Algorithm~\ref{alg:lagcae_online}. The only difference is that the nonlinear decoder $\mathcal{G}_d$ is replaced by a linear map based on the global basis $\Phi_r$ learned offline.



\section{Experiments}\label{sec:experiment}
This section reports three numerical experiments for representative transport-dominated parametrized PDEs. Unless otherwise stated, all reported errors are measured in the Eulerian frame on a fixed grid, so that different ROMs in the Eulerian and Lagrangian frames can be compared in the same output representation. In all experiments, we assess extrapolative prediction performance. Specifically, the ROMs are evaluated on parameter values not observed during training, and the learned reduced dynamics are used to predict the solution over a test time interval beyond the training time window. The relative $L^2$ error is defined by
\begin{equation}
\mathrm{err} = \frac{1}{n_{t^*}n_{\mu^*}}\sum_{t^*,\mu^*}\frac{\|u_E(t^*,\mu^*)-\hat{u}_E(t^*,\mu^*)\|_2}{\|u_E(t^*,\mu^*)\|_2}
\end{equation}
where $n_{t^*}$ and $n_{\mu^*}$ denote the numbers of test time instances and test parameters, respectively. Here, $u_E$ denotes the ground-truth solution in the Eulerian frame, and $\hat{u}_E$ denotes the interpolated Eulerian solution obtained from the Lagrangian prediction via $\hat{u}_E=\text{Interp}(\hat{u},\hat{x})$. The interpolation from the Lagrangian frame to the Eulerian frame is performed using radial basis function interpolation in 1D and linear interpolation in 2D.

We consider both linear and nonlinear ROMs: (i) Eulerian pDMD and Eulerian CAE--pDMD, and (ii) Lag--pDMD and LagCAE--pDMD. As the primary objective is to evaluate reduction efficiency and long-time prediction accuracy, rather than to perform extensive hyperparameter tuning, no separate validation set is introduced. All hyperparameters are prescribed a priori and kept fixed throughout, with no selection or adjustment based on the test data. 

For a fair comparison, the network architecture and training setup are kept identical for the Eulerian and Lagrangian CAEs, except for the input/output channels associated with the two frames. The neural network architecture and training hyperparameters are detailed in Appendix~\ref{app:arch-cae}.

\subsection{1D viscous Burgers equation}\label{subsec:1d_burgers}
The first benchmark problem is the one-dimensional viscous Burgers equation, which serves as a representative nonlinear transport-dominated problem featuring wave steepening and moving fronts:
\[
\frac{\partial u}{\partial t} + u\frac{\partial u}{\partial x} = \frac{1}{Re}\frac{\partial^2 u}{\partial x^2}, 
\qquad u(x,0) = u_0(x), 
\qquad x \in [0,1.5], \ t \in [0,4].
\]
The exact solution is available \cite{Maulik2021-cae-lstm}:
\[
u(x,t;Re) = \frac{\frac{x}{t+1}}{1 + \sqrt{\frac{t+1}{t_0}} \exp\!\left(Re\frac{x^2}{4t+4}\right)}, 
\qquad t_0 = \exp\!\left(\frac{Re}{8}\right),
\]
where $Re$ denotes the Reynolds number.

The spatial domain $[0,1.5]$ is discretized by a uniform grid with $128$ points. For training, we sample $21$ Reynolds numbers uniformly from $[200,600]$ and collect $81$ solution snapshots uniformly in time over $t\in[0,3.2]$. The test set contains four unseen Reynolds numbers $\{277,315,413,572\}$, each evaluated at $20$ time instances in $t\in(3.2,4.0]$. 

Figure~\ref{fig:1d_burgers_solution} compares predicted solution profiles at $t=4.0$ for $Re=315$ and $Re=572$. The Eulerian baselines (pDMD and CAE--pDMD) exhibit noticeable nonphysical oscillations near the peak and the steep front. In contrast, the Lagrangian variants (Lag--pDMD and LagCAE--pDMD) more closely match the front location and transition width of the ground truth, while largely suppressing spurious oscillations.

\begin{figure}[htbp]
    \centering
    \begin{subfigure}[b]{0.48\textwidth}
        \includegraphics[width=\textwidth]{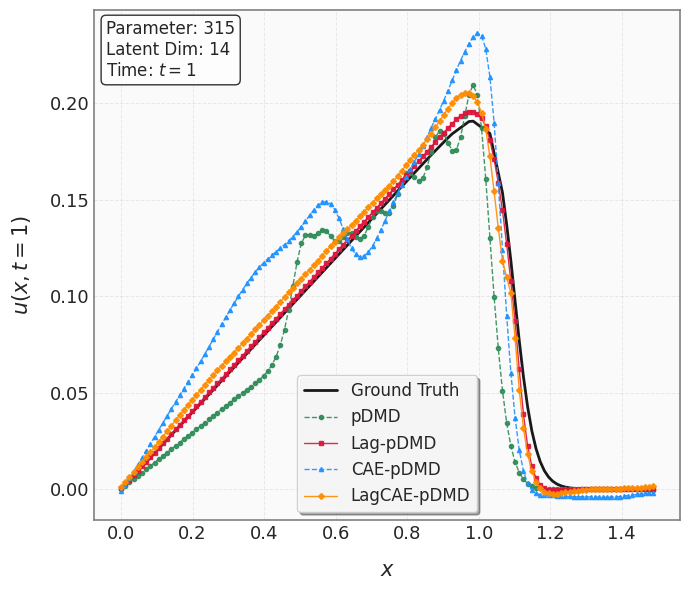}
        \caption{$Re=315$}
    \end{subfigure}
    \hfill
    \begin{subfigure}[b]{0.48\textwidth}
        \includegraphics[width=\textwidth]{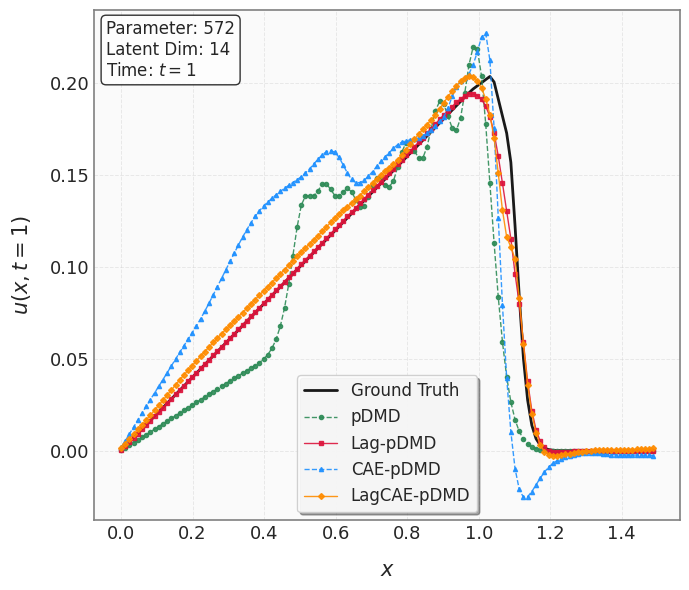}
        \caption{$Re=572$}
    \end{subfigure}
    \caption{Solution profiles of the 1D viscous Burgers equation at $t=4.0$. Results are shown for two unseen Reynolds numbers in the test set.}
    \label{fig:1d_burgers_solution}
\end{figure}

Figure~\ref{fig:1d_burgers_err_summary} summarizes the test errors of four ROMs (pDMD, CAE--pDMD, Lag--pDMD, and LagCAE--pDMD) on the 1D viscous Burgers benchmark, measured by the average relative $L^2$ error.

\begin{figure}[htbp]
    \centering
    \begin{subfigure}[ht]{0.32\textwidth}
        \includegraphics[width=\linewidth]{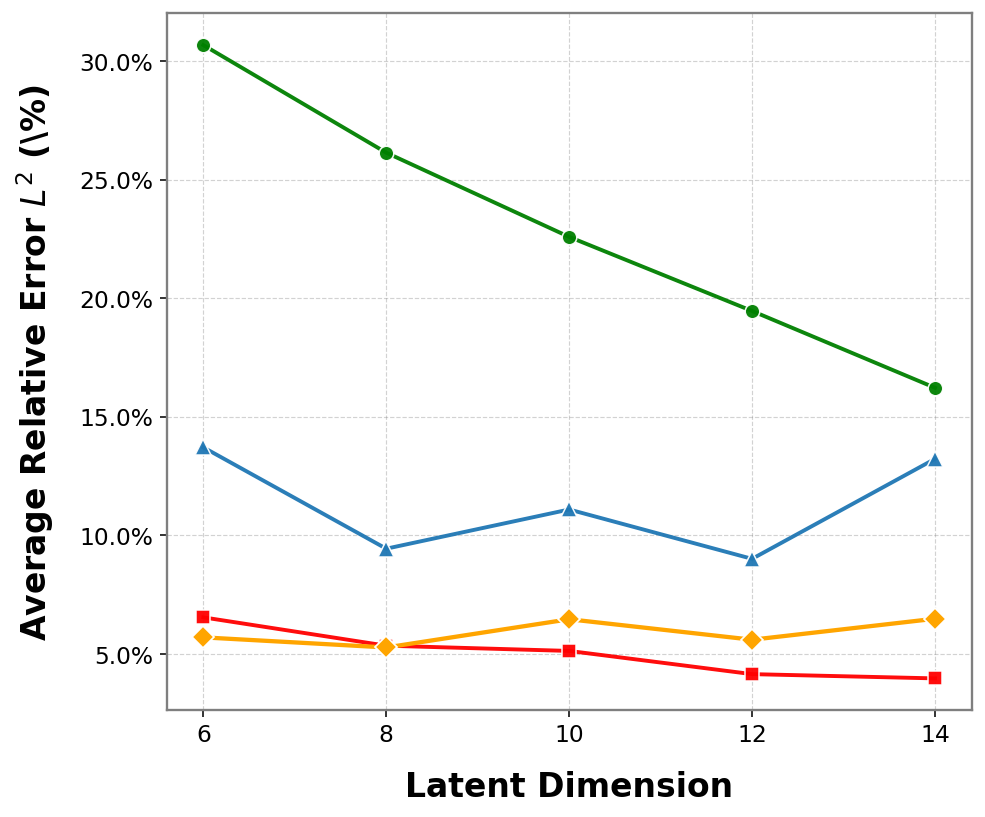}
        \caption{}\label{subfig:1d_burgers_err_ld}
    \end{subfigure}
    \hfill
    \begin{subfigure}[ht]{0.32\textwidth}
        \includegraphics[width=\linewidth]{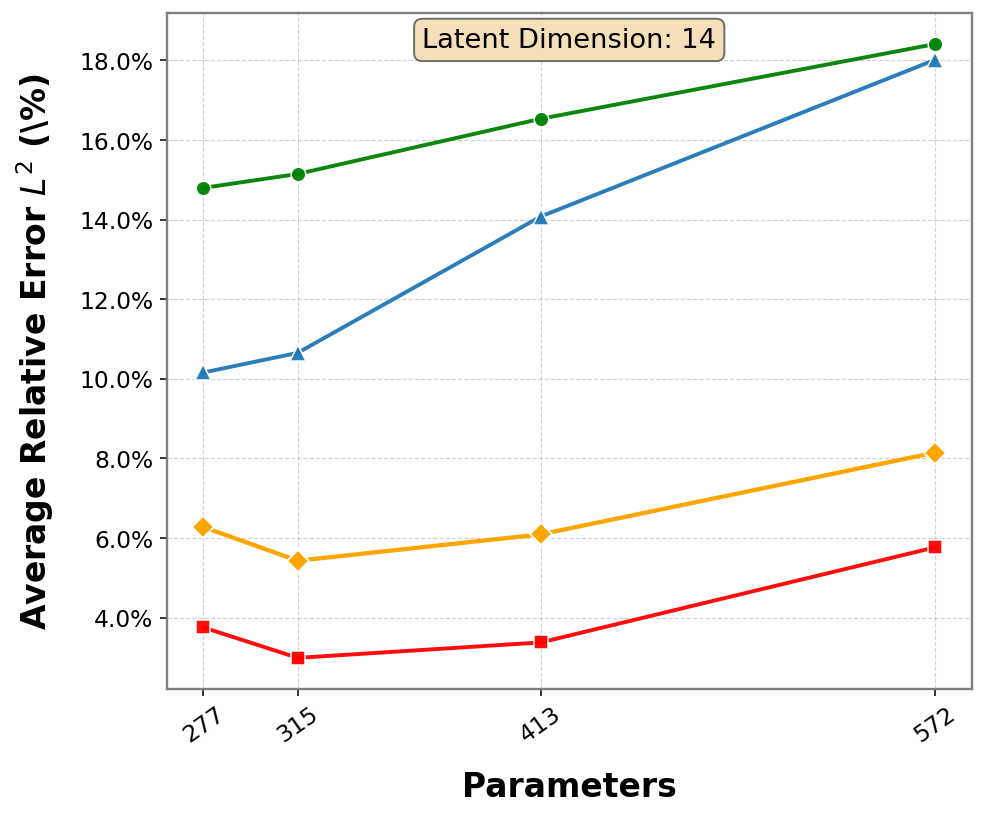}
        \caption{}\label{subfig:1d_burgers_err_param}
    \end{subfigure}
    \hfill
    \begin{subfigure}[ht]{0.32\textwidth}
        \includegraphics[width=\linewidth]{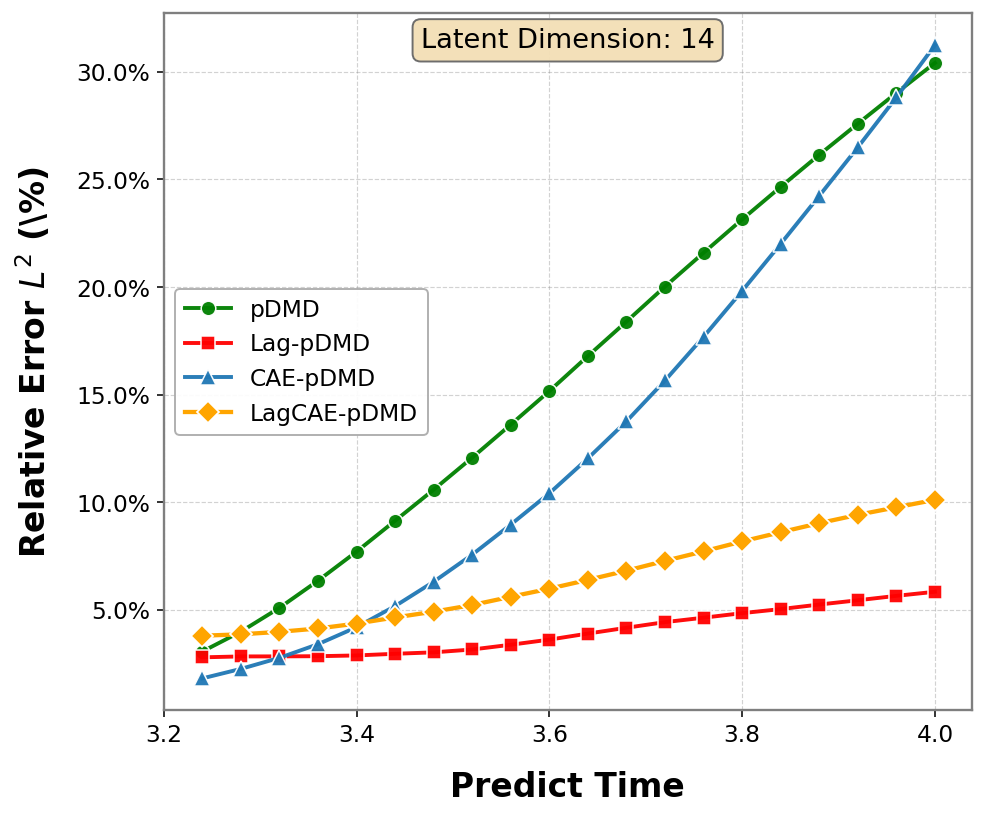}
        \caption{}\label{subfig:1d_burgers_err_time}
    \end{subfigure}
    \caption{1D viscous Burgers equation. Average relative $L^2$ error (\%), i.e., for unseen parameter values and future time instances:
    (a) error versus latent dimension $r$; 
    (b) error versus Reynolds number for a fixed latent dimension $r=14$;
    (c) error versus prediction step for $r=14$ during a 20-step prediction over the test time window.}
    \label{fig:1d_burgers_err_summary}
\end{figure}

\subheading{Effect of latent dimension.}
Figure~\ref{fig:1d_burgers_err_summary}\subref{subfig:1d_burgers_err_ld} reports the average test error (i.e., for unseen parameter values and future time instances) versus the latent dimension $r\in\{6,8,10,12,14\}$ (rank-$r$ truncation for DMD-based models and latent size $r$ for CAE-based models). Both Lagrangian approaches yield consistently lower errors across all $r$. The Eulerian baselines benefit from increasing $r$, but remain significantly less accurate, indicating that neither enlarging the linear subspace (pDMD) nor introducing a nonlinear encoder (CAE--pDMD) is sufficient to eliminate dominant errors in this setting.

\subheading{Error distribution of unseen Reynolds numbers.}
To isolate parameter effects, we fix $r=14$ and evaluate the methods at the four unseen Reynolds numbers $\{277,315,413,572\}$. As shown in Figure~\ref{fig:1d_burgers_err_summary}\subref{subfig:1d_burgers_err_param}, the prediction error increases with $Re$ for all methods, consistent with sharper fronts and reduced diffusion at higher $Re$. Nevertheless, Lag--pDMD and LagCAE--pDMD maintain substantially smaller errors than their Eulerian counterparts over the entire test parameter range, indicating weaker sensitivity to the Reynolds number and better robustness as the front becomes sharper.

\subheading{Prediction stability.}
Figure~\ref{fig:1d_burgers_err_summary}\subref{subfig:1d_burgers_err_time} shows the prediction error averaged over the test Reynolds numbers in the test window with fixed rank $r=14$. The Eulerian models exhibit rapid error accumulation as prediction proceeds, whereas the Lagrangian models show significantly slower error growth. This demonstrates the superior prediction stability of Lagrangian ROMs.

\subsection{2D advection--diffusion equation}
We next consider a two-dimensional parametric advection--diffusion equation on $(x,y,t)\in[0,4]\times[0,4]\times[0,1]$:

\begin{subequations}\label{eq:2d_advdiff}
\begin{align}
&\frac{\partial u}{\partial t}
+\cos(\theta)\frac{\partial u}{\partial x}
+\sin(\theta)\frac{\partial u}{\partial y}
=
D\left(
\frac{\partial^2 u}{\partial x^2}
+\frac{\partial^2 u}{\partial y^2}
\right),
\label{eq:2d_advdiff_a}\\
&u(x,y,0)
=
\exp\!\left(
-\frac{(x-2)^2+(y-2)^2}{0.1}
\right),
\label{eq:2d_advdiff_b}\\
&u(0,y,t)=u(4,y,t), \qquad
u(x,0,t)=u(x,4,t).
\label{eq:2d_advdiff_c}
\end{align}
\end{subequations}
where $\theta \in [0,2\pi]$ controls the advection direction and $D=0.001$ is the fixed diffusion coefficient.

Reference solutions are computed using a second-order central difference discretization for the Laplacian, an upwind scheme for the advection term, and forward Euler time integration. The spatial domain is discretized on a uniform $40\times 40$ grid with $\Delta x=\Delta y=0.1$ ($M=1600$ degrees of freedom). The time interval $[0,1]$ is discretized with time step $\Delta t=0.01$, yielding 101 time instances, which satisfies the CFL stability requirement for this setting. The training set consists of $30$ uniformly sampled parameters $\theta\in[0,2\pi]$, each associated with $81$ snapshots from $t\in[0,0.8]$. The test set contains $6$ unseen parameters $\mathcal{D}_{\text{test}}=\left\{\frac{k}{7}\cdot 2\pi\right\}_{k=1}^6$, each paired with $20$ snapshots from $t\in(0.8,1.0]$, thereby evaluating temporal extrapolation beyond the training time.

Figure~\ref{fig:2d_advdiff_solution} shows representative predictions for the parameter $\theta={10\pi}/7$ with a fixed latent dimension $r = 6$. The Eulerian baselines (pDMD and CAE--pDMD) exhibit noticeable artifacts, including spurious oscillations and peak splitting, leading to an incorrect spatial structure compared with the reference solution. In contrast, the Lagrangian variants (Lag--pDMD and LagCAE--pDMD) preserve a single coherent peak at the correct location and produce smoother fields, resulting in closer agreement with the reference.

\begin{figure}[htbp]
    \centering
    \begin{subfigure}[c]{0.3\textwidth}
        \centering
        \includegraphics[width=\textwidth]{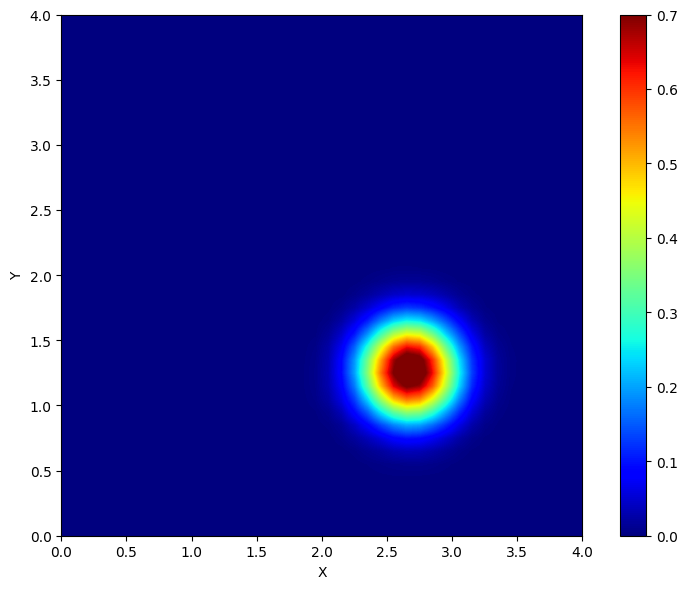}
        \caption{Ground truth}\label{subfig:2d_advdiff_gt}
    \end{subfigure}
    \hfill
    \begin{subfigure}[c]{0.3\textwidth}
        \centering
        \includegraphics[width=\textwidth]{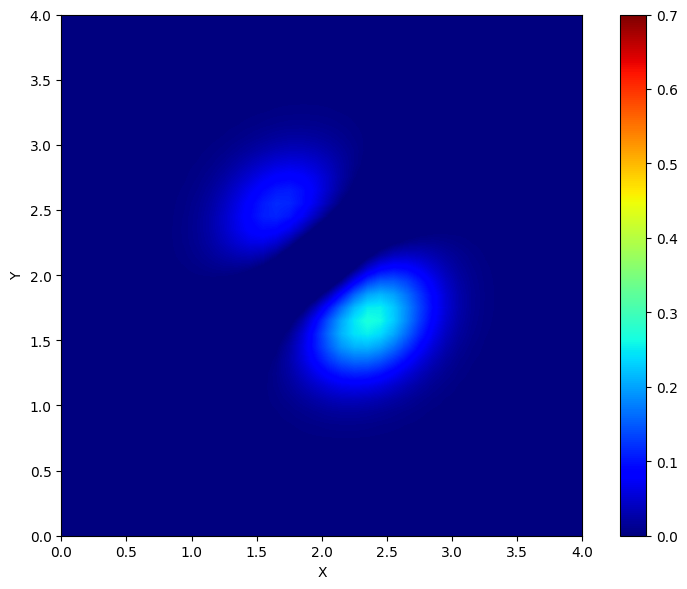}
        \caption{Eulerian DMD}\label{subfig:2d_advdiff_euler_dmd}
        \vspace{0.5cm}
        \includegraphics[width=\textwidth]{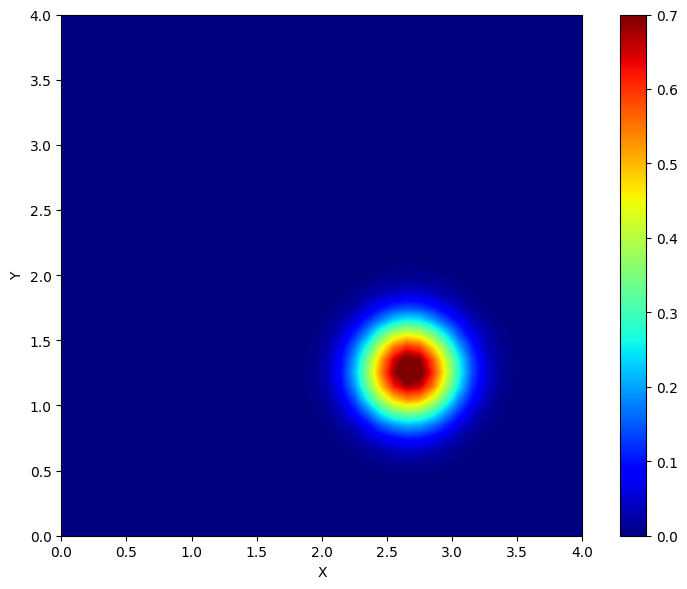}
        \caption{Lagrangian DMD}\label{subfig:2d_advdiff_lag_dmd}
    \end{subfigure}
    \hfill
    \begin{subfigure}[c]{0.3\textwidth}
        \centering
        \includegraphics[width=\textwidth]{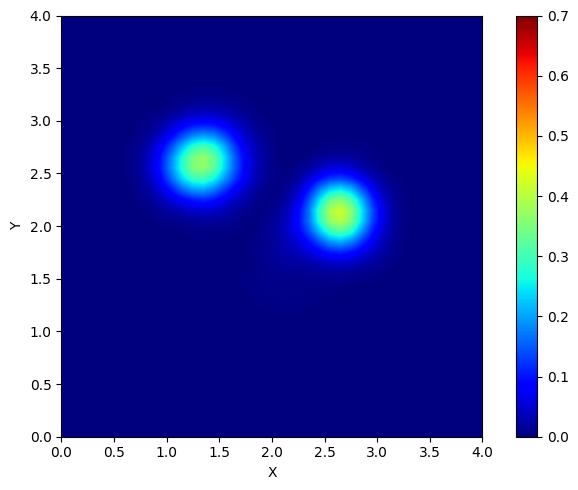}
        \caption{Eulerian CAE}\label{subfig:2d_advdiff_euler_cae}
        \vspace{0.5cm}
        \includegraphics[width=\textwidth]{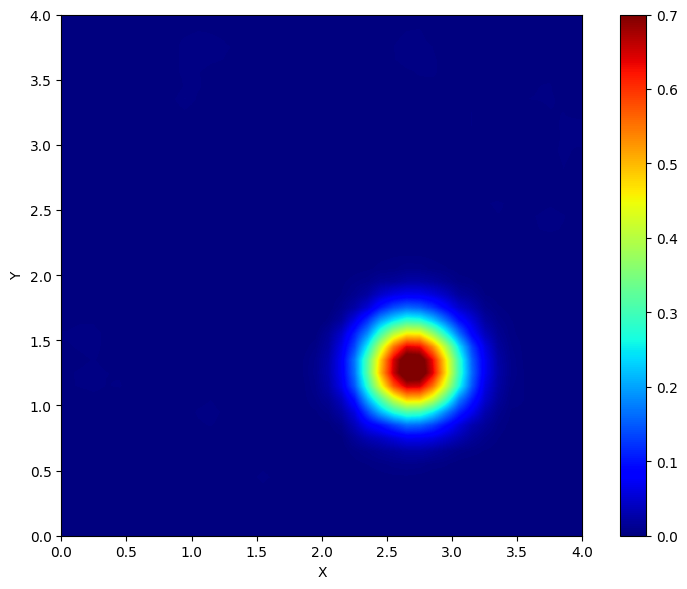}
        \caption{Lagrangian CAE}\label{subfig:2d_advdiff_lag_cae}
    \end{subfigure}
    \caption{Solutions of the 2D advection--diffusion equation for parameter $\theta={10\pi}/{7}$ at $t=1.0$.}
    \label{fig:2d_advdiff_solution}
\end{figure}

Figure~\ref{fig:2d_advdiff_err_summary} summarizes the test errors of four ROMs (pDMD, CAE--pDMD, Lag--pDMD, and LagCAE--pDMD) in the 2D advection--diffusion problem, measured by the average relative $L^2$ error.

\begin{figure}[ht]
    \begin{subfigure}[ht]{0.32\textwidth}
        \includegraphics[width=\linewidth]{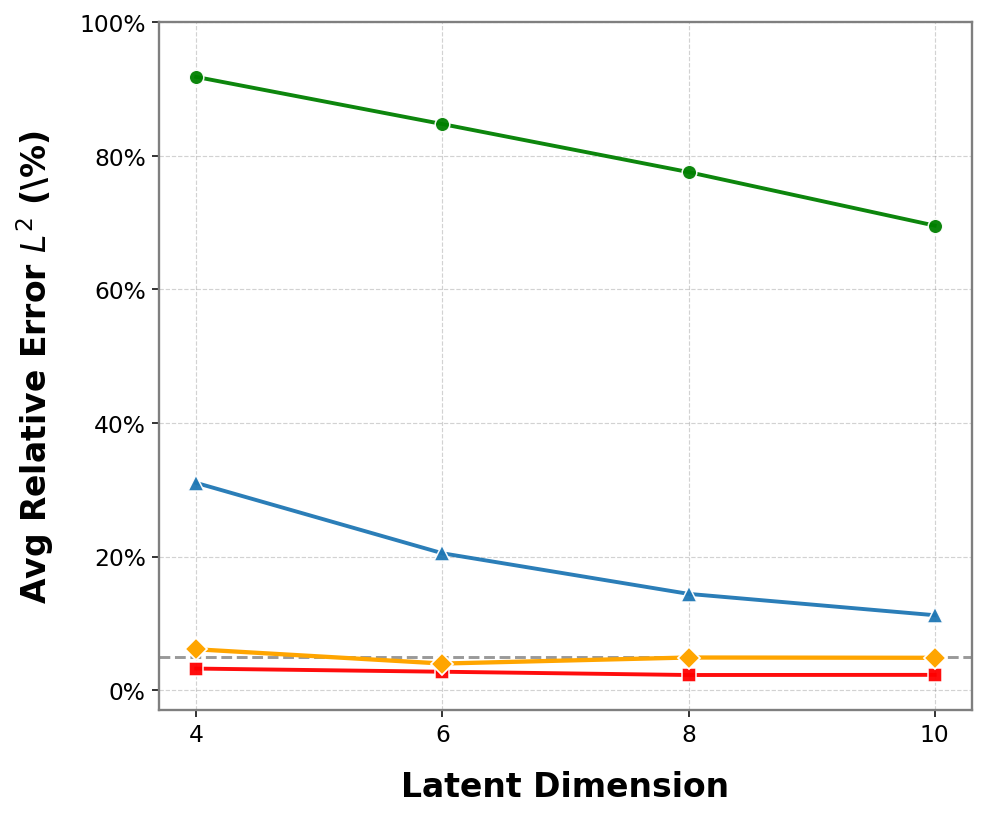}
    \caption{}\label{subfig:2d_advdiff_err_ld}
    \end{subfigure}
    \hfill
    \begin{subfigure}[ht]{0.32\textwidth}
        \includegraphics[width=\linewidth]{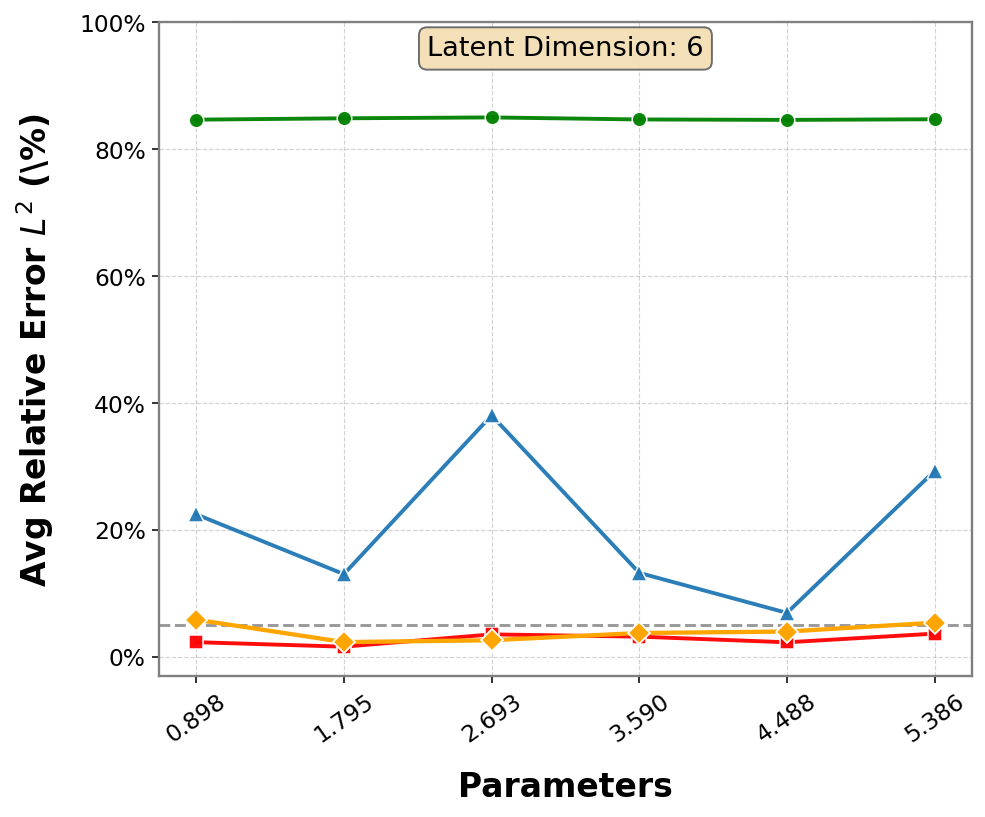}
    \caption{}\label{subfig:2d_advdiff_err_param}
    \end{subfigure}
    \hfill
     \begin{subfigure}[ht]{0.32\textwidth}
        \includegraphics[width=\linewidth]{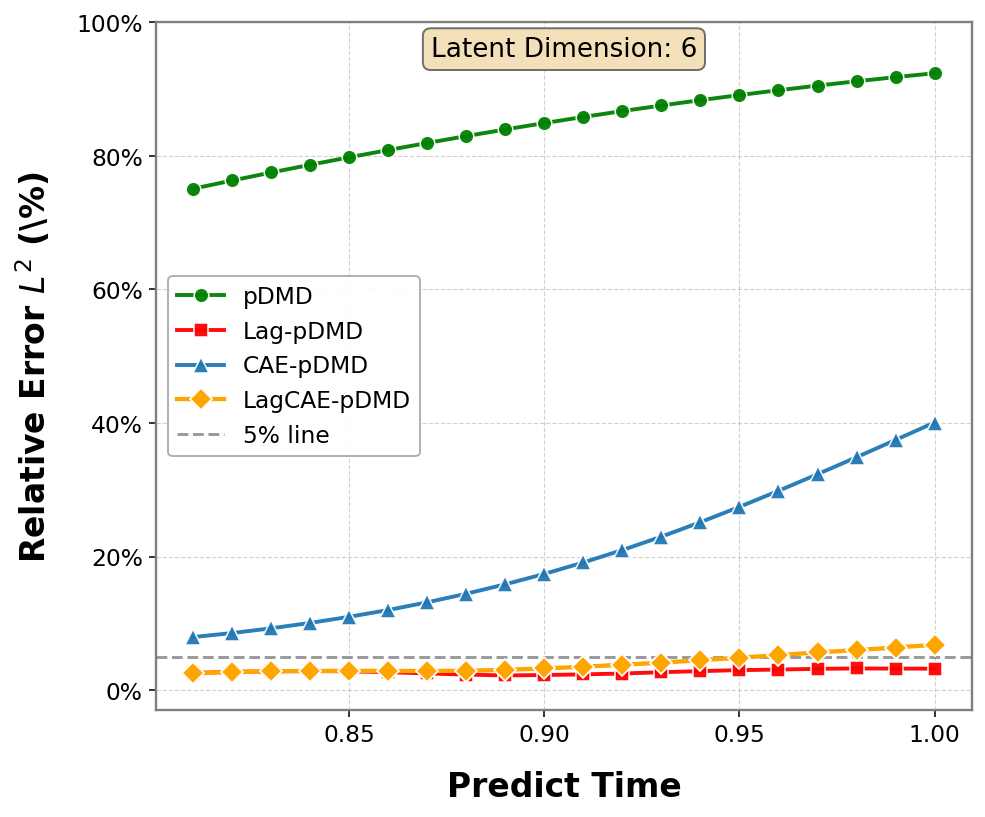}
        \caption{}\label{subfig:2d_advdiff_err_time}
    \end{subfigure}
    \caption{2D advection--diffusion equation. Average relative $L^2$ error (\%), i.e., for unseen parameter values and future time instances:
    (a) error versus latent dimension $r$; 
    (b) error versus the advection direction parameter $\theta$ for a fixed latent dimension $r=6$;
    (c) error versus prediction step for $r=6$ during a 20-step prediction over the test time window.}\label{fig:2d_advdiff_err_summary}
\end{figure}

\subheading{Effect of latent dimension.}
We evaluate all methods with latent dimensions $r\in\{4,6,8,10\}$. Figure~\ref{fig:2d_advdiff_err_summary}\subref{subfig:2d_advdiff_err_ld} reports the average relative $L^2$ error on the test set (i.e., for unseen parameter values and future time instances). Both Lagrangian methods achieve consistently low errors (below $5\%$) across all latent dimensions and exhibit comparable accuracy to each other. In contrast, the Eulerian baselines are substantially less accurate: pDMD remains above $80\%$, while CAE--pDMD reduces the error but still stays around $15\%$. This gap is consistent with the fact that the Eulerian solution is dominated by a translating profile; representing multiple shifted versions of the same structure with a fixed low-dimensional basis is inefficient and typically leads to prediction errors and spurious oscillations at small $r$. By explicitly tracking characteristic trajectories, the Lagrangian ROMs achieve significantly improved accuracy.

\subheading{Error distribution of unseen parameters.}
For a specific latent dimension ($r = 6$), Figure ~\ref{fig:2d_advdiff_err_summary}\subref{subfig:2d_advdiff_err_param} illustrates the parameter-dependent error distribution. All methods exhibit broadly consistent behavior across the test parameters, with no pronounced bias toward any specific parameter, except for CAE--pDMD. 

\subheading{Prediction stability.}
Figure~\ref{fig:2d_advdiff_err_summary}\subref{subfig:2d_advdiff_err_time} shows the average relative error for $r=6$ over a 20-step prediction in the test time window. The Eulerian methods accumulate errors rapidly as the prediction proceeds, whereas the Lagrangian methods maintain low error levels throughout the prediction horizon. This indicates that the Lagrangian ROMs can more reliably and stably predict future solutions.

\subsection{2D Burgers equation}\label{subsec:2d_burgers}
Finally, we consider a parametrized two-dimensional viscous Burgers equation, which serves as a more challenging nonlinear benchmark with coupled transport effects in higher dimensions. Reference solutions are obtained by numerically solving the following system with viscosity $\nu = 0.01$ over the parameter domain $\mu \in \mathcal{D} = [0.4, 0.8]$. 

\[
\left\{\begin{aligned}
& \frac{\partial u}{\partial t}+u\frac{\partial u}{\partial x}+v\frac{\partial u}{\partial y}=\nu(\frac{\partial^2 u}{\partial x^2}+\frac{\partial^2 u}{\partial y^2}) \\
& \frac{\partial v}{\partial t}+u\frac{\partial v}{\partial x}+v\frac{\partial v}{\partial y}=\nu(\frac{\partial^2 v}{\partial x^2}+\frac{\partial^2 v}{\partial y^2}) \\
& u(x,y,0;\mu)=\mu\sin(\pi(x-0.2))\sin(\pi (y-0.2))\chi_{[0.2,1.2]^2}+1, &(x,y)\in[0,5]^2 \\
& v(x,y,0;\mu)=\mu\sin(\pi(x-0.2))\sin(\pi (y-0.2))\chi_{[0.2,1.2]^2}+1, &(x,y)\in[0,5]^2 \\
& u(0,y,t)=u(5,y,t), u(x,0,t)=u(x,5,t), & t\in[0,2] \\
& v(0,y,t)=v(5,y,t), v(x,0,t)=v(x,5,t), & t\in[0,2] \\
\end{aligned}\right.
\]

The spatial domain is discretized using a uniform $128 \times 128$ grid, and the reference solution is computed using the forward Euler scheme with a time step size of $\Delta t = 5 \times 10^{-3}$. Snapshots are uniformly recorded every four time steps, yielding 101 snapshots for each parameter instance.

The training set consists of $17$ parameters uniformly distributed in $\mathcal{D}$, namely $\mathcal{D}_{\text{train}}=\{0.4+0.025\,i\}_{i=0}^{16}$. For each training parameter, we use the first $91$ snapshots to construct the ROM surrogate. The test set contains $8$ parameters randomly sampled from $\mathcal{D}$; for each test parameter, the remaining $10$ snapshots are used to evaluate future prediction accuracy.

Figure~\ref{fig:2d_burgers_solution} compares representative predictions from different methods with a fixed latent dimension of $r=12$. The Eulerian baselines exhibit noticeable spurious oscillations and fail to accurately track the transported structures, whereas the Lagrangian variants better preserve the transport behavior and produce predictions that agree more closely with the reference solution.

\begin{figure}[htbp]
    \centering
    \begin{subfigure}[c]{0.3\textwidth}
        \centering
        \includegraphics[width=\textwidth]{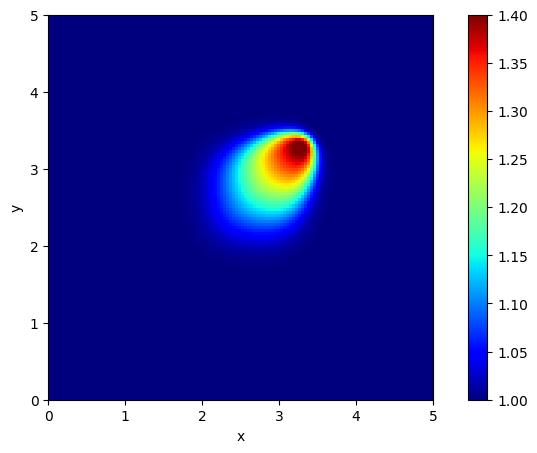}
        \caption{Ground truth}\label{subfig:2d_burgers_gt}
    \end{subfigure}
    \hfill
    \begin{subfigure}[c]{0.3\textwidth}
        \centering
        \includegraphics[width=\textwidth]{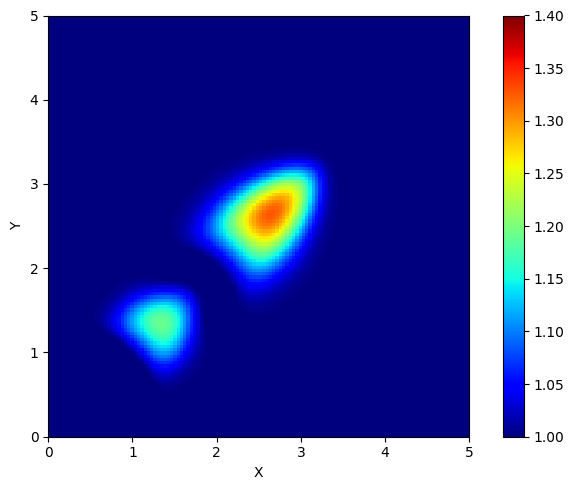}
        \caption{Eulerian DMD}\label{subfig:2d_burgers_euler_dmd}
        \vspace{0.5cm}
        \includegraphics[width=\textwidth]{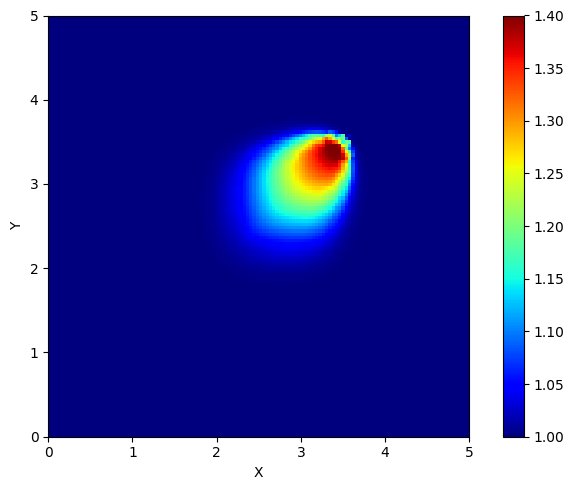}
        \caption{Lagrangian DMD}\label{subfig:2d_burgers_lag_dmd}
    \end{subfigure}
    \hfill
    \begin{subfigure}[c]{0.3\textwidth}
        \centering
        \includegraphics[width=\textwidth]{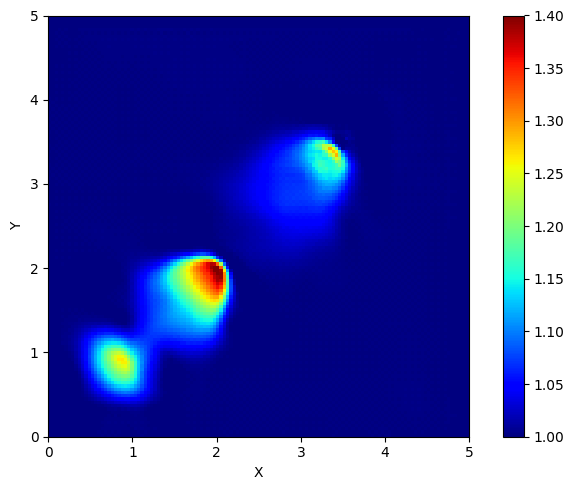}
        \caption{Eulerian CAE}\label{subfig:2d_burgers_euler_cae}
        \vspace{0.5cm}
        \includegraphics[width=\textwidth]{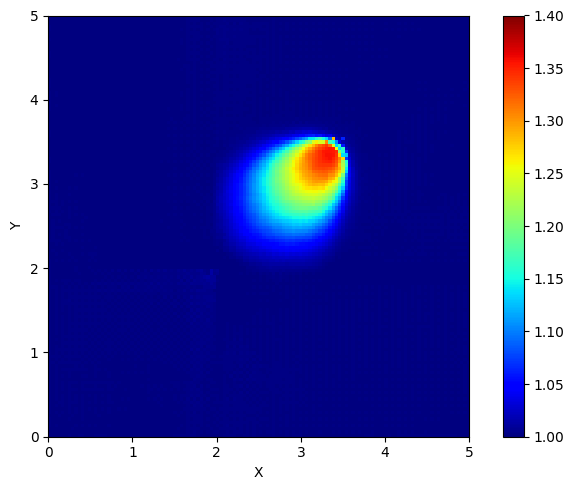}
        \caption{Lagrangian CAE}\label{subfig:2d_burgers_lag_cae}
    \end{subfigure}
    \caption{2D viscous Burgers equation with parameter $\mu=0.7345$ at time $t=2.0$}\label{fig:2d_burgers_solution}
\end{figure}

Figure~\ref{fig:2d_burgers_err_summary} summarizes the test errors of four ROMs (pDMD, CAE--pDMD, Lag--pDMD, and LagCAE--pDMD) in the 2D Burgers problem, measured by the average relative $L^2$ error.

\subheading{Effect of the latent dimension.}
We evaluate the performance across multiple latent space dimensions $\{6, 8, 10, 12, 14, 16, 18, 20\}$. Figure~\ref{fig:2d_burgers_err_summary}\subref{subfig:2d_burgers_err_ld} provides a comparative analysis of prediction errors (i.e., for unseen parameter values and future time instances) among different approaches. Our results indicate that the Lagrangian-based methods consistently outperform their Eulerian counterparts across all tested latent dimensions. In all tested cases, the relative errors of the Lagrangian ROMs remain below $1\%$. Although errors in Eulerian methods decrease with increasing latent dimensions, they still result in relatively high prediction errors.

\subheading{Error distribution of unseen parameters.}
Figure~\ref{fig:2d_burgers_err_summary}\subref{subfig:2d_burgers_err_param} examines the parameter-dependent prediction errors for all methods with a fixed latent dimension of $r=12$. A consistent increase in error with rising parameter values is observed across all methods, attributed to amplified transport phenomena at higher parameter values. The increase is most pronounced for the Eulerian baselines, especially pDMD, while Lag--pDMD remains below $0.5\%$ over the entire test range. LagCAE--pDMD also maintains low errors (below $1\%$) but shows a mild upward trend at larger $\mu$.

\subheading{Prediction stability.}
Figure~\ref{fig:2d_burgers_err_summary}\subref{subfig:2d_burgers_err_time} shows the average relative error in future predictions outside the training time window. The Eulerian methods accumulate error rapidly as the prediction proceeds, with CAE--pDMD showing the fastest growth. In contrast, Lag--pDMD and LagCAE--pDMD exhibit only a mild increase in error in the prediction.

\begin{figure}[ht]
    \centering
    \begin{subfigure}[ht]{0.32\textwidth}
        \includegraphics[width=\linewidth]{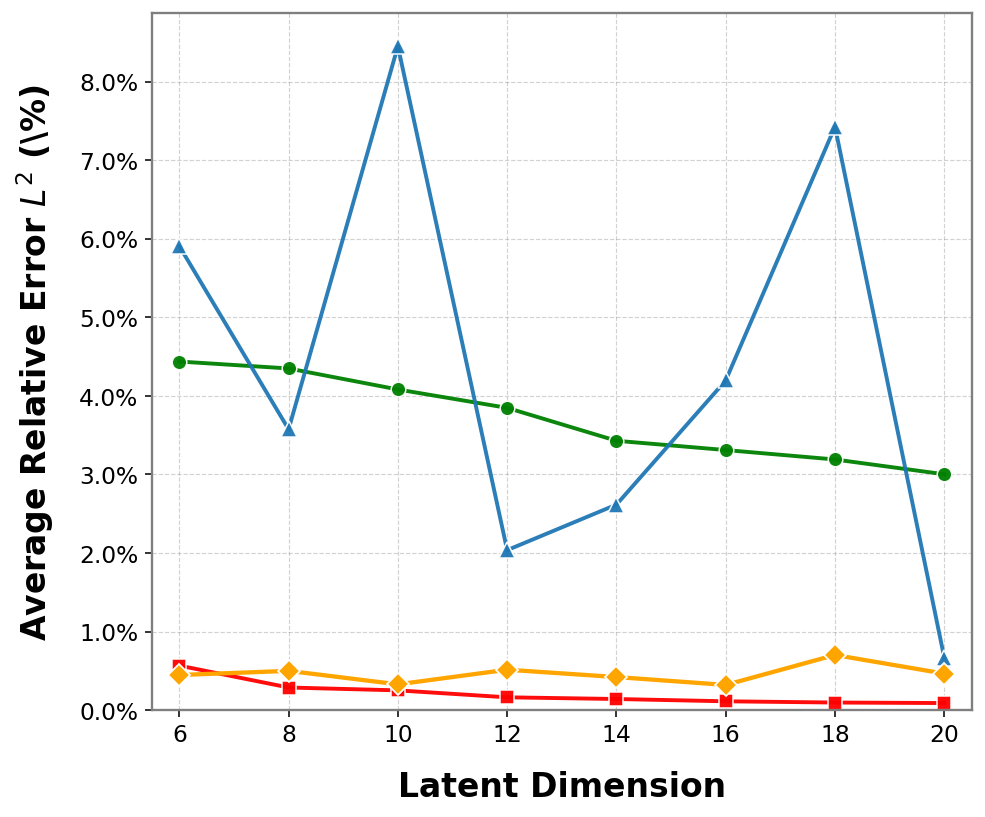}
        \caption{}\label{subfig:2d_burgers_err_ld}
    \end{subfigure}
    \hfill
    \begin{subfigure}[ht]{0.32\textwidth}
        \includegraphics[width=\linewidth]{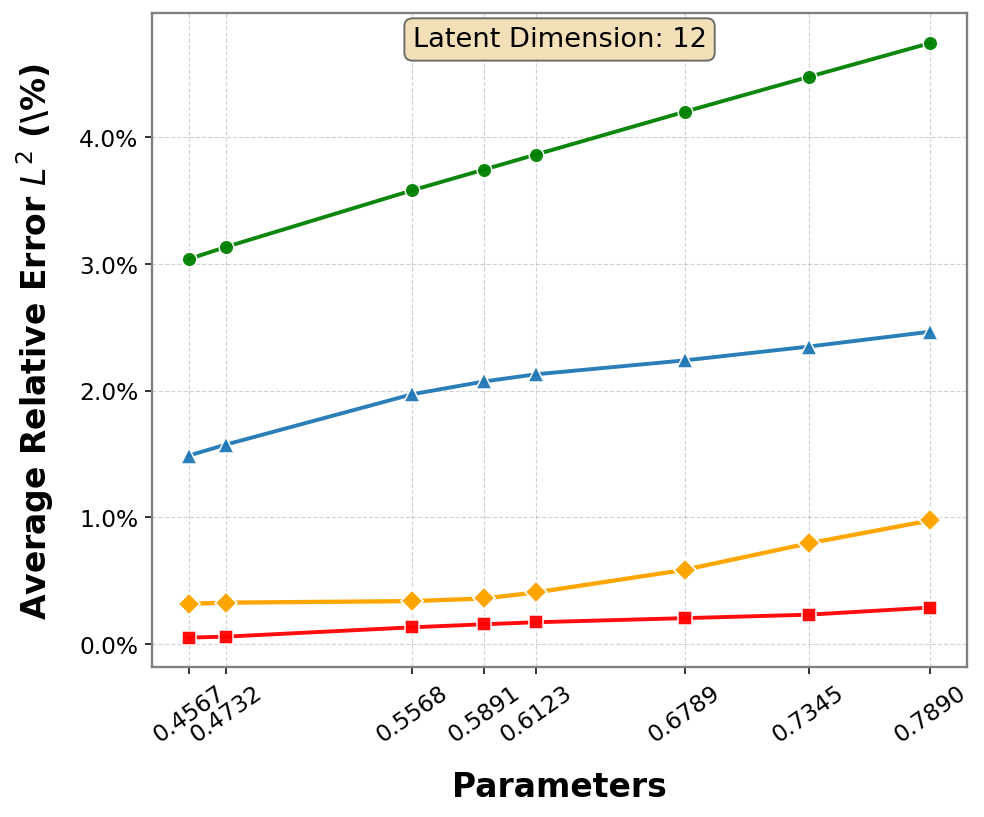}
        \caption{}\label{subfig:2d_burgers_err_param}
    \end{subfigure}
    \hfill
    \begin{subfigure}[ht]{0.32\textwidth}
        \includegraphics[width=\linewidth]{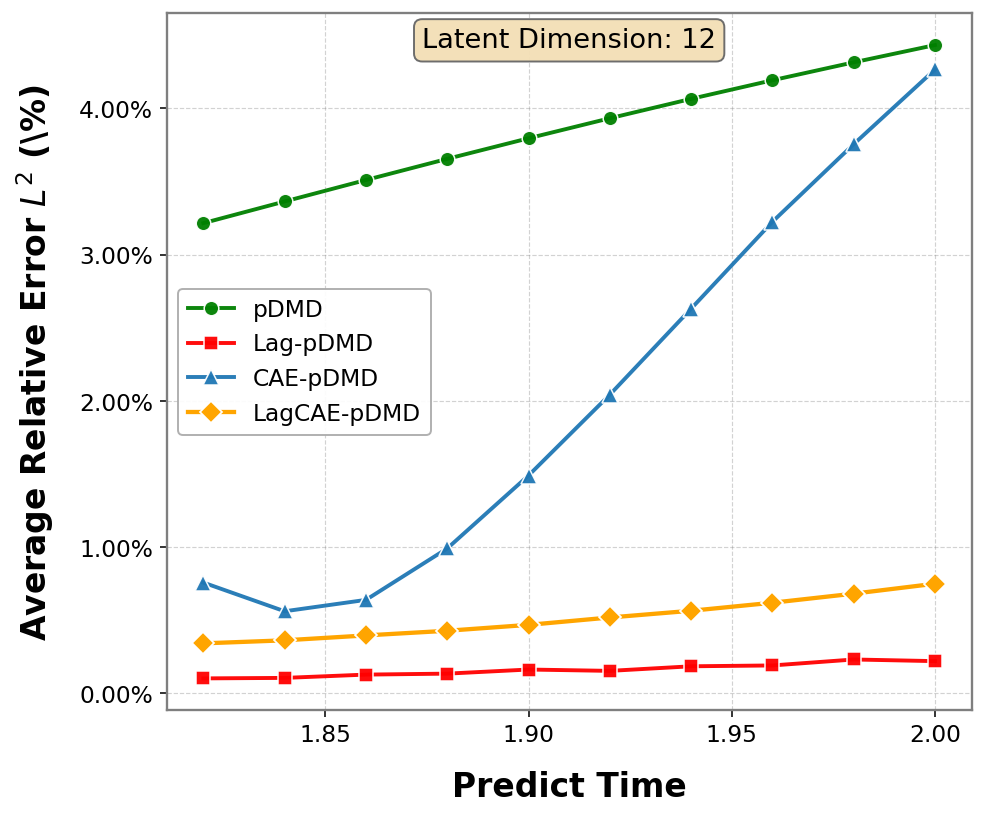}
        \caption{}\label{subfig:2d_burgers_err_time}
    \end{subfigure}
    \caption{2D viscous Burgers equation. Average relative $L^2$ error (\%), i.e., for unseen parameter values and future time instances:
    (a) error versus latent dimension $r$; 
    (b) error versus the parameter $\mu$ for a fixed latent dimension $r=12$; 
    (c) error versus prediction step for $r=12$ during a 10-step prediction over the test time window.}
    \label{fig:2d_burgers_err_summary}
\end{figure}


\section{Conclusion}\label{sec:conclusion}
In this paper, we show that even nonlinear autoencoder-based ROMs in the Eulerian frame, despite their strong compression capability, may fail to provide accurate and robust long-time predictions for transport-dominated problems. This limitation arises from the weak correlations induced by spatial shifts in the Eulerian representation.

To address this issue, we develop two Lagrangian reduced-order models: a non-intrusive Lagrangian autoencoder-based ROM and a Lagrangian parametric DMD. In addition, we prove that the Kolmogorov $n$-width of the solution manifold decays faster in the Lagrangian frame. We also show  that solutions exhibit significantly stronger correlations in the Lagrangian frame. These properties provide explanations for the improved approximation and prediction performance of Lagrangian ROMs.

Numerical experiments confirm that the proposed Lagrangian ROMs significantly outperform their Eulerian counterparts in predicting future solutions of parametrized transport-dominated systems, achieving both higher accuracy and improved robustness.

Several directions merit further investigation. First, extending the approach to problems with strong deformations and shocks calls for robust remapping and regularization strategies. Second, richer latent dynamics models beyond linear DMD may further enhance the prediction stability. 


\section*{Acknowledgements}
\textbf{Funding} The work of Z. Peng was partially supported by the Hong Kong Research Grants Council grants Early Career Scheme 26302724 and General Research Fund 16306825. The work of Y. Xiang was partially supported by the National Key Research and Development Program of China (No. 2025YFA1016800) and the Project of Hetao Shenzhen-HKUST Innovation Cooperation Zone HZQB-KCZYB-2020083.

\section*{Data Availability}
Data availability The data that support the findings of this study are available from the corresponding author upon reasonable request.

\section*{Declarations}
\textbf{Conflict of interest} The authors have no relevant financial or non-financial interests to disclose.

\bibliographystyle{spmpsci}      
\bibliography{ref}   

\appendix
\section*{Appendix}
\section{Architectures of the convolutional autoencoders}\label{app:arch-cae}
\begin{table}[ht]
\centering
\small
\caption{Eulerian/Lagrangian CAE architecture for 1D Burgers. The weight decay used in training is $10^{-10}$ and the coefficient of gradient loss $\lambda_{\mathrm{grad}}=0.05$. All convolutional and transpose-convolutional layers use a kernel size of $5$, stride $2$, and zero padding $2$.}
\label{tab:1d_burgers_cae_arch}
\begin{tabular}{lcc}
\toprule
\multicolumn{3}{l}{\textbf{encoder}} \\
\midrule
layer & input shape & output shape \\
\midrule
Conv1d with SiLU & $1\;/\; 2\times128$   & $32\times64$ \\
Conv1d with SiLU & $32\times64$    & $32\times32$ \\
Conv1d with SiLU & $32\times32$    & $32\times16$ \\
Conv1d with SiLU & $32\times16$    & $32\times8$   \\
Conv1d with SiLU & $32\times8$    & $32\times4$   \\
Flatten          & $32\times4$     & $128$               \\
Linear with SiLU & $128$                   & $40$                 \\
Linear with SiLU & $40$                     & $8$                  \\
\midrule
\multicolumn{3}{l}{\textbf{decoder}} \\
\midrule
layer & input shape & output shape \\
\midrule
Linear with SiLU          & $8$                   & $40$                 \\
Linear with SiLU          & $40$                  & $128$               \\
Unflatten                 & $128$                & $32\times4$   \\
ConvTranspose1d with SiLU & $32\times4$   & $32\times8$\\
ConvTranspose1d with SiLU & $32\times8$   & $32\times16$ \\
ConvTranspose1d with SiLU & $32\times16$ & $32\times32$ \\
ConvTranspose1d with SiLU & $32\times32$ & $32\times64$ \\
ConvTranspose1d with SiLU & $32\times64$ & $1\;/\; 2\times128$ \\
\bottomrule
\end{tabular}
\end{table}


\begin{table}[ht]
\centering
\small
\caption{Eulerian/Lagrangian CAE architecture for 2D advection--diffusion equation. The weight decay used in training is $10^{-11}$ and the coefficient of gradient loss $\lambda_{\mathrm{grad}}=0$. All convolutional and transpose-convolutional layers use a kernel size of $4$, stride $2$, and circular padding $1$.}\label{tab:2d_advdiff_cae_arch}
\begin{tabular}{lcc}
\toprule
\multicolumn{3}{l}{\textbf{encoder}} \\
\midrule
layer & input shape & output shape \\
\midrule
Conv2d with SiLU & $1\;/\;3\times40\times40$   & $16\times20\times20$ \\
Conv2d with SiLU & $16\times20\times20$    & $16\times10\times10$ \\
Conv2d with SiLU & $16\times10\times10$    & $16\times5\times5$ \\
Flatten          & $16\times5\times5$      & $400$               \\
Linear with SiLU & $400$                   & $90$                 \\
Linear with SiLU & $90$                     & $8$                  \\
\midrule
\multicolumn{3}{l}{\textbf{decoder}} \\
\midrule
layer & input shape & output shape \\
\midrule
Linear with SiLU        & $8$                   & $90$                 \\
Linear with SiLU        & $90$                  & $400$               \\
Unflatten               & $400$                & $16\times5\times5$   \\
ConvTranspose2d with SiLU & $16\times5\times5$   & $16\times10\times10$ \\
ConvTranspose2d with SiLU & $16\times10\times10$ & $16\times20\times20$ \\
ConvTranspose2d with SiLU & $16\times20\times20$ & $1\;/\;3\times40\times40$ \\
\bottomrule
\end{tabular}
\end{table}


\begin{table}[ht]
\centering
\small
\caption{Eulerian/Lagrangian CAE architecture for 2D Burgers. The weight decay used in training is $10^{-8}$ and the coefficient of gradient loss $\lambda_{\mathrm{grad}}=0.05$. All convolutional and transpose-convolutional layers use a kernel size of $5$, stride $2$, and zero padding $2$.}\label{tab:2d_burgers_cae_arch}
\begin{tabular}{lcc}
\toprule
\multicolumn{3}{l}{\textbf{encoder}} \\
\midrule
layer & input shape & output shape \\
\midrule
Conv2d with SiLU & $2\;/\;4\times128\times128$   & $32\times64\times64$ \\
Conv2d with SiLU & $32\times64\times64$    & $32\times32\times32$ \\
Conv2d with SiLU & $32\times32\times32$    & $32\times16\times16$ \\
Conv2d with SiLU & $32\times16\times16$    & $32\times8\times8$   \\
Flatten          & $32\times8\times8$      & $2048$               \\
Linear with SiLU & $2048$                   & $100$                 \\
Linear with SiLU & $100$                     & $8$                  \\
\midrule
\multicolumn{3}{l}{\textbf{decoder}} \\
\midrule
layer & input shape & output shape \\
\midrule
Linear with SiLU        & $8$                   & $100$                 \\
Linear with SiLU        & $100$                  & $2048$               \\
Unflatten               & $2048$                & $32\times8\times8$   \\
ConvTranspose2d with SiLU & $32\times8\times8$   & $32\times16\times16$ \\
ConvTranspose2d with SiLU & $32\times16\times16$ & $32\times32\times32$ \\
ConvTranspose2d with SiLU & $32\times32\times32$ & $32\times64\times64$ \\
ConvTranspose2d with SiLU & $32\times64\times64$ & $2\;/\;4\times128\times128$ \\
\bottomrule
\end{tabular}
\end{table}

\end{document}